\documentclass[a4paper,11pt]{article}
\usepackage[utf8]{inputenc}
\usepackage{a4wide}
\usepackage{algorithm}
\usepackage{algorithmic}
\usepackage{latexsym,amsfonts,amsmath,amssymb,mathrsfs,url,amsthm}
\usepackage{aliascnt}
\usepackage{mathtools}
\usepackage{dsfont}
\usepackage{color,graphicx}
\usepackage{flafter}
\usepackage{placeins}
\usepackage{lipsum}
\usepackage{subcaption}
\usepackage{hyperref}
\usepackage{cleveref}
\usepackage{xcolor}

\newtheorem{theorem}{Theorem}[section]

\newaliascnt{lemma}{theorem}
\newtheorem{lemma}[lemma]{Lemma}
\aliascntresetthe{lemma}

\newaliascnt{example}{theorem}
\newtheorem{example}[example]{Example}
\aliascntresetthe{example}

\newaliascnt{remark}{theorem}
\newtheorem{remark}[remark]{Remark}
\aliascntresetthe{remark}

\newaliascnt{definition}{theorem}
\newtheorem{definition}[definition]{Definition}
\aliascntresetthe{definition}

\newaliascnt{corollary}{theorem}
\newtheorem{corollary}[corollary]{Corollary}
\aliascntresetthe{corollary}

\newaliascnt{proposition}{theorem}
\newtheorem{proposition}[proposition]{Proposition}
\aliascntresetthe{proposition}

\crefname{theorem}{theorem}{theorems}
\crefname{lemma}{lemma}{lemmas}
\crefname{example}{example}{examples}
\crefname{remark}{remark}{remarks}
\crefname{definition}{definition}{definitions}
\crefname{corollary}{corollary}{corollaries}
\crefname{proposition}{proposition}{propositions}

\mathtoolsset{showonlyrefs}

\newcommand{\bsa}{\boldsymbol{a}}
\newcommand{\bsx}{\boldsymbol{x}}
\newcommand{\bsy}{\boldsymbol{y}}
\newcommand{\bsz}{\boldsymbol{z}}
\newcommand{\EE}{\mathbb{E}}
\newcommand{\NN}{\mathbb{N}}
\newcommand{\PP}{\mathbb{P}}
\newcommand{\RR}{\mathbb{R}}
\newcommand{\Ncal}{\mathcal{N}}
\newcommand{\Xcal}{\mathcal{X}}
\newcommand{\Ycal}{\mathcal{Y}}
\DeclareMathOperator{\supp}{supp}

\allowdisplaybreaks

\title{Quasi-uniform designs from random candidates: optimal candidate complexity and farthest-point sampling\thanks{The authors contributed equally to this work.}}
\author{Takashi Goda\thanks{Graduate School of Engineering, The University of Tokyo, 7-3-1 Hongo, Bunkyo-ku, Tokyo 113-8656, Japan (\url{goda@frcer.t.u-tokyo.ac.jp}; \url{jokokun@g.ecc.u-tokyo.ac.jp})} \and Hengjun Xu\footnotemark[2]}
\date{\today}

\begin{document}

\maketitle

\begin{abstract}
    We study randomized constructions of quasi-uniform designs on a compact metric-measure space satisfying two-sided polynomial ball-growth conditions. Given integers $N\le M$, we first draw $M$ independent candidate points and then retain the first $N$ points of a farthest-point traversal of the candidate set. We prove probabilistic non-asymptotic bounds on the mesh ratio and show that $M=\Theta(N\log N)$ is the sharp order of the candidate-pool size required for bounded mesh ratio: a sufficiently large multiple of $N\log(N/\delta)$ suffices with probability at least $1-\delta$, whereas, if $M=o(N\log N)$, the mesh ratio diverges in probability for every procedure that selects $N$ points from the same independent candidate pool. If $M/(N\log M)\to\infty$, the upper bound on the mesh ratio for the exact farthest-point sampling (FPS) tends to $2$. We also quantify the effect of approximate FPS and analyze direct and fast implementations in spaces of bounded doubling dimension. Finally, using a single infinite stream of candidate points and increasing candidate budgets, we construct an almost surely quasi-uniform nested sequence; under supercritical oversampling and exact FPS, or more generally when the approximation factors tend to $1$, its mesh-ratio limit superior equals $2$ for compact positive-volume subsets of Euclidean space.
\end{abstract}


\section{Introduction}\label{sec:introduction}

Finite point sets with good geometric distribution properties are fundamental in scattered data approximation, kernel interpolation, meshless numerical methods, and the design of computer experiments \cite{FLS06,SWN03,SW06,W05}. Let $(\Xcal,d)$ be a compact metric space and let $X_N=\{\bsx_1,\ldots,\bsx_N\}\subset\Xcal.$
The covering and separation radii of $X_N$ in $\Xcal$ are defined by
\[
h_{\Xcal}(X_N)
:=
\sup_{\bsx\in\Xcal}\min_{1\le i\le N}d(\bsx,\bsx_i)\quad \text{and}\quad q(X_N)
:=
\frac{1}{2}\min_{1\le i<j\le N}d(\bsx_i,\bsx_j),
\]
respectively. The ratio
\[
\rho_{\Xcal}(X_N)
:=
\frac{h_{\Xcal}(X_N)}{q(X_N)}
\]
is called the mesh ratio. A small covering radius ensures that the domain is well covered, while a large separation radius prevents undesirable clustering of the points. A sequence of point sets is called \emph{quasi-uniform} if its mesh ratios are uniformly bounded.

Quasi-uniformity provides a natural balance between the minimax and maximin distance criteria that are commonly used in space-filling design \cite{JMY90,J16,MM95,P17,PM12}. It is also an important geometric condition in scattered data approximation: the covering radius determines the approximation scale, while the separation radius affects numerical stability and conditioning. Consequently, quasi-uniform point sets play a central role in radial basis function and kernel approximation, as well as in meshless methods for partial differential equations \cite{DSW05,SW06,W05,WSH21}. On an $s$-dimensional regular domain, a quasi-uniform family has both covering and separation radii of the optimal order $N^{-1/s}$.

A simple and attractive procedure for generating well-distributed points is the greedy-packing algorithm. Starting from an arbitrary point $\bsx_1\in\Xcal$, it recursively chooses
\[
\bsx_{n+1}
\in
\operatorname*{arg\,max}_{\bsx\in\Xcal}d(\bsx,X_n) 
\quad \text{with}\quad 
d(\bsx,X_n):=\min_{1\le i\le n}d(\bsx,\bsx_i).
\]
For a finite metric space, this construction is usually referred to as a \emph{farthest-first traversal}, a \emph{greedy permutation}, or \emph{farthest-point sampling} (FPS). It is closely related to the classical greedy algorithm of Gonzalez for metric clustering \cite{EHS20,G85}. Pronzato and Zhigljavsky \cite{PZ23} recently established a sharp quasi-uniformity result for greedy packing. In particular, every point set generated by the exact continuous-domain algorithm has mesh ratio at most~$2$. They also proved that, for every nested sequence in a compact positive-volume subset of Euclidean space, the limit superior of the mesh ratio is at least~$2$. Thus the uniformity constant~$2$ attained by greedy packing is optimal among all nested constructions.

Despite this strong theoretical property, the continuous-domain algorithm is not immediately implementable on a general domain. At each iteration, one must identify a center of a largest empty ball relative to the current point set and the boundary of the domain. This can be accomplished by exploiting Voronoi geometry in certain special settings; see, for example, \cite{XG26} for triangular domains. For a general compact domain, however, a direct and geometry-independent procedure for finding the exact maximizer may be unavailable or computationally expensive.

A natural practical alternative is to replace the continuous search space by a finite candidate set. Given $\Ycal_M=\{Y_1,\ldots,Y_M\}\subset\Xcal$, one recursively chooses
\[
\bsx_{n+1}
\in
\operatorname*{arg\, max}_{\bsy\in\Ycal_M}d(\bsy,X_n)
\]
and stops after $N$ points have been selected. This finite-candidate version is easy to implement and requires only evaluations of the metric. Although Pronzato and Zhigljavsky \cite{PZ23} showed, in a deterministic setting, that the performance on the entire domain can be controlled by the covering radius $h_{\Xcal}(\Ycal_M)$ of the candidate set, the result does not specify how the candidate set should be generated or how large it must be relative to the desired output size~$N$.

In this paper, we generate the candidate set randomly. More precisely, we equip $\Xcal$ with a Borel probability measure $\mu$, draw $Y_1,\ldots,Y_M\overset{\mathrm{iid}}{\sim}\mu$, and apply exact or approximate FPS to the resulting candidate cloud. The construction requires only a sampling procedure for $\mu$, a distance oracle, and the integers $M$ and $N$. In particular, it avoids repeatedly solving a continuous largest-empty-ball problem and applies without an explicit parametrization of the boundary.

Our main assumptions concern the small-ball behavior of the sampling measure. We suppose that there exist $s>0$ and positive constants $c_-$, $c_+$, and $r_0$, such that
\[
c_{-}r^s
\le
\mu(B(\bsx,r))
\le
c_{+}r^s
\]
holds for any $\bsx\in\Xcal$ and $0<r\le r_0$, where $B(\bsx,r)$ denotes the closed metric ball with its center at $\bsx$ and radius $r$. This includes compact Ahlfors-regular metric-measure spaces, regular full-dimensional Euclidean domains, and compact manifolds with their natural volume measures. For $M$ independent random points on such a space, the
covering radius has the characteristic order
\[
h_{\Xcal}(\Ycal_M)
\asymp
\left(\frac{\log M}{M}\right)^{1/s},
\]
see \cite{J87,RS16} for classical and general results on random coverings. Since the natural covering-radius scale of a quasi-uniform $N$-point design is $N^{-1/s}$, balancing the two scales yields the candidate-complexity threshold $M\asymp N\log N$.

This logarithmic factor has the same local-occupancy origin as connectivity thresholds in random geometric graphs. At the target resolution $r\asymp N^{-1/s}$, an $s$-regular space contains on the order of $N$ essentially disjoint regions, each having probability mass on the order of $N^{-1}$. Eliminating all empty regions therefore has the coupon-collector scale $M\asymp N\log N$. Equivalently, the critical local-occupancy parameter is $Mr^s\asymp\log M$, which is also the scale at which isolated vertices disappear and connectivity emerges in standard random geometric graphs \cite{P99,PY25}. The events are nevertheless different: our analysis requires coverage of the ambient space rather than connectivity among the sampled vertices.

Now, the main contributions of this paper are summarized as follows.

\begin{enumerate}
\item
We establish non-asymptotic probabilistic bounds on the mesh ratio for exact FPS. In a simplified form, our result states that there exists a constant $C>0$, depending only on the regularity constants of
$(\Xcal,d,\mu)$, such that, with probability at least $1-\delta$,
\[
    \max_{2\le n\le N}\rho_{\Xcal}(X_n)
    \le
    \frac{2}{1-C\left(\dfrac{N\log(M/\delta)}{M}\right)^{1/s}}
\]
provided that the denominator is positive. In particular, a sufficiently
large multiple of $N\log(N/\delta)$ candidates is enough to obtain a bounded mesh ratio with probability at least $1-\delta$. If $M/(N\log M)\to\infty$, the resulting mesh-ratio upper bound tends to~$2$.

\item
We prove that the order $N\log N$ is sharp within the iid candidate model.
More precisely, with probability tending to~$1$,
\[
\inf_{Z_N\subseteq\Ycal_M,\, |Z_N|=N}
\rho_{\Xcal}(Z_N)
\ge
c\left(\frac{N\log M}{M}\right)^{1/s}
\]
for a constant $c>0$. This lower bound applies to every possible procedure
that selects $N$ points from the candidate cloud, and not only to FPS.
Consequently, if $N\le M=o(N\log N)$, then the mesh ratio diverges in probability for every candidate-thinning
algorithm. Hence $M=\Theta(N\log N)$ is the sharp candidate-pool order for
constructing point sets with bounded mesh ratio from iid candidates.

\item
We extend the analysis to approximate FPS. Suppose that, for some $0<\alpha\le1$, the point selected at iteration $n$ satisfies
\[
d(\bsx_{n+1},X_n)
\ge
\alpha
\max_{\bsy \in\Ycal_M}d(\bsy,X_n).
\]
We show that the mesh-ratio bound for exact FPS is enlarged essentially by the factor $1/\alpha$. This separates the error caused by the random discretization of the domain from the error caused by the approximate farthest-point search.

A direct implementation of exact FPS uses $O(MN)$ distance evaluations and $O(M)$ storage, which becomes $O(N^2\log N)$ work at the critical iid candidate size. In spaces of bounded doubling dimension, near-linear algorithms are available for constructing exact or approximate greedy permutations \cite{EHS20,HM06}. Combining such algorithms with our candidate-complexity result leads, for fixed approximation accuracy and intrinsic dimension, to an end-to-end cost of order $O(N\log^2N)$.

\item
Finally, we develop a nested version of our randomized construction. We use a single infinite iid stream $Y_1,Y_2,\ldots$ and an increasing sequence of candidate budgets $M_n$. At step $n$, the next point is selected approximately farthest from the current design among the first $M_n$ candidates. A sufficiently large choice $M_n\asymp n\log n$ yields an almost surely quasi-uniform nested sequence. If $M_n/(n\log n)\to \infty$ and the approximation factors $\alpha_n$ tend to~$1$, then
\[
\limsup_{n\to\infty}\rho_{\Xcal}(X_n)\le2
\quad\text{almost surely}.
\]
For compact positive-volume subsets of Euclidean space, the lower bound of Pronzato and Zhigljavsky then implies that this limit superior is exactly~$2$.
\end{enumerate}

The mechanism behind the upper bounds is simple: the random candidate cloud provides a sufficiently fine discretization of the domain, and FPS converts this covering property into simultaneous covering and separation guarantees for the selected subset. The matching lower bound shows, however, that the logarithmic oversampling factor is not merely an artifact of the analysis or of the greedy selection rule. It is caused by the largest empty regions inherent in iid sampling and cannot be removed by using a different thinning algorithm. 

The remainder of the paper is organized as follows. In \Cref{sec:preliminaries}, we introduce the metric-measure assumptions, geometric quantities, and exact and approximate FPS rules. In \Cref{sec:exact}, we establish deterministic candidate-set estimates, derive non-asymptotic probabilistic bounds, and prove the sharp $N\log N$ candidate-complexity result. In \Cref{sec:approximate}, we study approximate FPS and its direct and fast implementations, and place FPS and maximal independent sets in a common metric-net framework. Although exact FPS is recovered from approximate FPS by setting all approximation factors equal to~$1$, we treat the exact and approximate cases separately. This entails some overlap between \Cref{sec:exact,sec:approximate}, but keeps the sharp candidate-complexity argument transparent and isolates the effect of inexact farthest-point searches. The nested randomized construction and its almost-sure properties are analyzed in \Cref{sec:nested}. Numerical experiments are presented in \Cref{sec:numerics}, and concluding remarks are given in \Cref{sec:conclusion}.

\section{Preliminaries}\label{sec:preliminaries}

Throughout this paper, $(\Xcal,d)$ denotes a compact metric space with positive diameter
\[
D_{\Xcal}:=\operatorname{diam}(\Xcal)>0,
\]
and $\mu$ denotes a non-atomic Borel probability measure with $\supp(\mu)=\Xcal$.  The full-support assumption is necessary for iid candidates to cover the prescribed domain, while non-atomicity ensures that independently sampled candidates are distinct with probability~$1$. For $\bsx\in\Xcal$ and $r\ge0$, we write
\[
B(\bsx,r):=\{\bsy \in\Xcal:d(\bsx,\bsy)\le r\}
\]
for the closed metric ball. For every nonempty set $A\subseteq\Xcal$, we also use the notation
\[
d(\bsx,A):=\inf_{\bsa\in A}d(\bsx,\bsa).
\]
For nonnegative quantities $a$ and $b$, we write $a\lesssim b$ if $a\le Cb$ for a constant $C$ independent of the candidate-pool size $M$ and the output size $N$; dependence on additional parameters is indicated by a subscript. The notation $a\asymp b$ means that both $a\lesssim b$ and $b\lesssim a$ hold.

\subsection{Metric-measure spaces and geometric regularity}
\label{subsec:regularity}

For $r>0$, define the lower and upper ball-mass functions by
\[
\underline{\mu}(r)
:=
\inf_{\bsx\in\Xcal}\mu(B(\bsx,r)),
\quad
\overline{\mu}(r)
:=
\sup_{\bsx\in\Xcal}\mu(B(\bsx,r)).
\]
Compactness and full support imply that $\underline{\mu}(r)>0$ for every
fixed $r>0$. We also denote by $\Ncal_{\Xcal}(r)$ the
covering number of $\Xcal$, namely the smallest integer $K$ for
which there exist $\bsz_1,\ldots,\bsz_K\in\Xcal$ satisfying
\[
\Xcal\subseteq\bigcup_{j=1}^K B(\bsz_j,r).
\]
Some of our estimates will first be stated in terms of
$\Ncal_{\Xcal}(r)$ and $\underline{\mu}(r)$ without any
polynomial-growth assumption. The sharp rates in $M$ and $N$ are then
obtained under the following regularity condition.

\begin{definition}[$s$-regular metric-measure space]
\label{def:s-regular}
Let $s>0$. We say that $(\Xcal,d,\mu)$ is \emph{$s$-regular} if there
exist constants $0<c_-\le c_+<\infty$ such that
\begin{equation}
c_-r^s
\le
\mu(B(\bsx,r))
\le
c_+r^s
\quad
\text{for all $\bsx\in\Xcal$ and $0<r\le D_{\Xcal}$.}
\label{eq:s-regularity}
\end{equation}
The exponent $s$ will be referred to as the intrinsic dimension of the
space.
\end{definition}

This condition is also called Ahlfors regularity; see, for instance, \cite[Chapter~8]{H01}. In terms of the ball-mass functions, it is equivalent to $c_-r^s \le \underline{\mu}(r) \le \overline{\mu}(r) \le c_+r^s$. The lower and upper bounds imply full support and non-atomicity, respectively. Moreover, an $s$-regular metric space has finite doubling dimension, with a doubling constant depending only on $c_-, c_+$, and $s$; see, for example, \cite{H01}. This fact will be useful when discussing fast FPS implementations.

\begin{remark}\label{rem:local-regularity}
It suffices to assume \eqref{eq:s-regularity} for $0<r\le r_0$ with fixed $r_0>0$; after changing the constants, the estimates extend to every $0<r\le D_{\Xcal}$. For $r_0<r\le D_{\Xcal}$, one may use
\[
\mu(B(\bsx,r))\le 1\le r_0^{-s}r^s\quad \text{and}\quad 
\mu(B(\bsx,r))
\ge c_-r_0^s
\ge c_-\left(\frac{r_0}{D_{\Xcal}}\right)^s r^s.
\]
Thus, the standing assumption is essentially a small-scale condition.
\end{remark}

\begin{example}[Euclidean domains]
Let $\Xcal\subset\RR^d$ be compact with positive Lebesgue measure, equipped with the Euclidean metric and normalized Lebesgue measure. Suppose that there exist $\kappa>0$ and $r_0>0$ such that
\[
    |\Xcal\cap B(\bsx,r)|
    \ge
    \kappa r^d
\]
for all $\bsx\in\Xcal$ and $0<r\le r_0$. Then $(\Xcal,\|\cdot\|_2,\mu)$ is $d$-regular. This condition is satisfied, for example, when $\Xcal$ is the closure of a bounded Lipschitz domain.
\end{example}

\begin{example}[Compact manifolds]
\label{ex:manifold}
Let $\Xcal$ be a compact connected $s$-dimensional Riemannian manifold without boundary, equipped with its geodesic distance and normalized Riemannian volume. Then $(\Xcal,d,\mu)$ is $s$-regular. Compact Riemannian manifolds are therefore included in the present framework; see also \cite{GL17,RS16}.
\end{example}

The regularity assumption gives the expected order $r^{-s}$ for the
covering number.

\begin{lemma}
\label{lem:covering-number}
Suppose that $(\Xcal,d,\mu)$ is $s$-regular. Then, for every
$0<r\le D_{\Xcal}$, we have
\[
\frac{1}{c_+r^s}
\le
\Ncal_{\Xcal}(r)
\le
\frac{2^s}{c_{-}r^s}.
\]
\end{lemma}

\begin{proof}
If $K$ balls of radius $r$ cover $\Xcal$, then we have
\[
1=\mu(\Xcal)
\le
\sum_{j=1}^K\mu(B(\bsz_j,r))
\le
Kc_+r^s,
\]
which proves the lower bound. For the upper bound, take an inclusion-wise maximal set $\{\bsz_1,\ldots,\bsz_K\}\subset\Xcal$ whose pairwise distances are greater than $r$. Maximality implies that the balls $B(\bsz_j,r)$ cover $\Xcal$, whereas the balls $B(\bsz_j,r/2)$ are pairwise disjoint. Hence
\[
1
\ge
\sum_{j=1}^K\mu(B(\bsz_j,r/2))
\ge
Kc_-\left(\frac r2\right)^s,
\]
and the upper bound follows.
\end{proof}

\subsection{Covering and separation radii, and mesh ratio}
\label{subsec:geometric-quantities}

Let $X_N=\{\bsx_1,\ldots,\bsx_N\}\subset\Xcal$ be a set of $N$ distinct
points. Its covering radius, also referred to as the \emph{fill distance}, in $\Xcal$ is
\begin{equation}
h_{\Xcal}(X_N)
:=
\sup_{\bsx\in\Xcal}d(\bsx,X_N).
\label{eq:fill-distance}
\end{equation}
More generally, for every nonempty $A\subseteq\Xcal$, we write
\[
h_A(X_N):=\sup_{\bsx\in A}d(\bsx,X_N).
\]
For $N\ge2$, the separation radius of $X_N$ is
\begin{equation}
q(X_N)
:=
\frac12\min_{1\le i< j\le N}d(\bsx_i,\bsx_j).
\label{eq:separation-radius}
\end{equation}
The mesh ratio is defined by
\begin{equation}
\rho_{\Xcal}(X_N)
:=
\frac{h_{\Xcal}(X_N)}{q(X_N)}.
\label{eq:mesh-ratio}
\end{equation}

\begin{definition}[Quasi-uniformity]
\label{def:quasi-uniformity}
A family $\{X_N\}_{N\in\mathcal{I}}$ of finite point sets, with $|X_N|=N$, is called \emph{quasi-uniform} in $\Xcal$ if
\[
\sup_{N\in\mathcal{I},\, N\ge2}
\rho_{\Xcal}(X_N)<\infty.
\]
A sequence $(\bsx_n)_{n\ge1}\subset\Xcal$ of distinct points is called quasi-uniform if its initial segments $X_N=\{\bsx_1,\ldots,\bsx_N\}$ form a quasi-uniform family.
\end{definition}

We shall compare the covering radius of a constructed point set with the optimal $N$-point covering radius
\begin{equation}
h_N^*(\Xcal)
:=
\inf_{\substack{Z\subset\Xcal,\, |Z|=N}}
h_{\Xcal}(Z).
\label{eq:optimal-fill-distance}
\end{equation}
The following elementary proposition records the natural geometric scale in an $s$-regular space.

\begin{proposition}
\label{prop:natural-scale}
Suppose that $(\Xcal,d,\mu)$ is $s$-regular. Then, for every $N\ge1$,
\begin{equation}
(c_+N)^{-1/s}
\le
h_N^*(\Xcal)
\le
2(c_-N)^{-1/s}.
\label{eq:optimal-fill-bounds}
\end{equation}
Moreover, every set $X_N\subset\Xcal$ of $N\ge2$ distinct points satisfies
\begin{equation}
q(X_N)
\le
(c_-N)^{-1/s}.
\label{eq:separation-upper-bound}
\end{equation}
\end{proposition}

\begin{proof}
Let $Z=\{\bsz_1,\ldots,\bsz_N\}\subset\Xcal$ and put $h=h_{\Xcal}(Z)$. The balls $B(\bsz_j,h)$ cover $\Xcal$, and hence
\[
1
\le
\sum_{j=1}^N\mu(B(\bsz_j,h))
\le
Nc_+h^s.
\]
Taking the infimum over $Z$ proves the lower bound on $h_N^*(\Xcal)$.

For the upper bound, put $r_N=2(c_-N)^{-1/s}$. If $r_N\le D_{\Xcal}$, \Cref{lem:covering-number} gives $\Ncal_{\Xcal}(r_N)\le N$, so there is a set of at most $N$ points with covering radius at most $r_N$. Additional distinct points can be inserted without increasing the covering radius. If $r_N>D_{\Xcal}$, the claimed estimate is immediate from $h_N^*(\Xcal)\le D_{\Xcal}<r_N$.

Finally, let $q=q(X_N)$. For every $0<r<q$, the balls $B(\bsx_1,r),\ldots,B(\bsx_N,r)$ are pairwise disjoint. Therefore
\[
1
\ge
\sum_{j=1}^N\mu(B(\bsx_j,r))
\ge
Nc_-r^s.
\]
Letting $r\uparrow q$ proves the upper bound on $q(X_N)$.
\end{proof}

The following equivalence follows immediately from \Cref{prop:natural-scale} and the definition of the mesh ratio, so we omit the proof.
\begin{corollary}
\label{cor:quasi-uniform-scales}
Under the assumptions of \Cref{prop:natural-scale}, a family
$\{X_N\}_{N\in\mathcal{I}}$ is quasi-uniform if and only if $h_{\Xcal}(X_N)\lesssim N^{-1/s}$ and $q(X_N)\gtrsim N^{-1/s}$ uniformly for $N\in\mathcal{I}$. In particular, quasi-uniformity forces
both quantities to have the optimal order $N^{-1/s}$.
\end{corollary}

\subsection{Random candidate sets}
\label{subsec:random-candidates}

Let $M\ge N$ and let $Y_1,\ldots,Y_M\overset{\mathrm{iid}}{\sim}\mu$. We denote the resulting candidate set by
\begin{equation}
\Ycal_M:=\{Y_1,\ldots,Y_M\}
\label{eq:candidate-set}
\end{equation}
and its covering radius in the original domain, which we call the \emph{covering error}, is given by
\begin{equation}
\varepsilon_M
:=
h_{\Xcal}(\Ycal_M)
=
\sup_{\bsx\in\Xcal}d(\bsx,\Ycal_M).
\label{eq:candidate-covering-error}
\end{equation}
Since $\mu$ is non-atomic, the candidate points are pairwise distinct with probability~$1$. The random variable $\varepsilon_M$ is the covering radius of the iid candidate cloud. Its order and tail probabilities will be studied in \Cref{sec:exact}; see \cite{RS16} for general results on random covering radii in regular metric-measure spaces.

All probabilistic statements below refer to the randomness of the candidate points unless additional algorithmic randomness is explicitly introduced.

\subsection{Exact and approximate farthest-point sampling}
\label{subsec:fps}

Let $\Ycal_M\subset\Xcal$ be a deterministic set of $M$ distinct candidate points and let $1\le N\le M$. Starting from a point $\bsx_1\in\Ycal_M$, define
\[
X_n:=\{\bsx_1,\ldots,\bsx_n\},
\quad 1\le n\le N,
\]
where we choose $\bsx_n\in \Ycal_M\setminus X_{n-1}$ in a certain way. For each $n\le M$, the residual covering radius of the candidate set is
\begin{equation}
R_n
:=
h_{\Ycal_M}(X_n)
=
\max_{\bsy\in\Ycal_M}d(\bsy,X_n).
\end{equation}
In particular, since $X_M=\Ycal_M$, we set $R_M=0$. For $n<M$, exact FPS chooses
\begin{equation}
\bsx_{n+1}
\in
\operatorname*{arg\,max}_{\bsy\in\Ycal_M}d(\bsy,X_n),
\label{eq:exact-fps}
\end{equation}
so that $d(\bsx_{n+1},X_n)=R_n$.

To include inexact or fast search procedures, let $\alpha_1,\ldots,\alpha_{N-1}\in(0,1]$. We say that the construction is an \emph{$(\alpha_n)$-approximate farthest-point sampling} if
\begin{equation}
d(\bsx_{n+1},X_n)
\ge
\alpha_nR_n,
\quad n=1,\ldots,N-1.
\label{eq:approximate-fps}
\end{equation}
For a constant $\alpha\in(0,1]$, an $\alpha$-approximate FPS means that \eqref{eq:approximate-fps} holds with $\alpha_n\ge\alpha$ at every step. Exact FPS is recovered by setting $\alpha_n=1$ for all $n$. We shall use
\[
\underline{\alpha}_n:=\min_{1\le j\le n}\alpha_j
\]
for the smallest approximation factor used during the first $n$ steps. 

Because the candidate set is finite, an exact maximizer in \eqref{eq:exact-fps} always exists, and hence a point satisfying \eqref{eq:approximate-fps} exists for every $\alpha_n\le1$. We fix an arbitrary deterministic tie-breaking rule whenever the maximizer is not unique. The initial point may be chosen arbitrarily; all estimates developed below are uniform in this choice.

For every ordered set of distinct points, $2q(X_n)=\min_{1\le j<n}d(\bsx_{j+1},X_j)$. Moreover, $R_n$ is nonincreasing in $n$. It follows from \eqref{eq:approximate-fps} that
\begin{equation}
q(X_n)
\ge
\frac{\underline{\alpha}_{n-1}}{2}R_{n-1},
\quad n\ge2.
\label{eq:approximate-separation-bound}
\end{equation}
For exact FPS, $d(\bsx_{j+1},X_j)=R_j$ for every $j$, and hence
\begin{equation}
q(X_n)=\frac12R_{n-1}.
\label{eq:exact-separation-identity}
\end{equation}

The candidate covering radius $\varepsilon_M$ measures the difference between covering the finite cloud and covering the whole domain. For every $X_n\subseteq\Ycal_M$,
\begin{equation}
R_n
\le
h_{\Xcal}(X_n)
\le
R_n+\varepsilon_M.
\label{eq:discretization-inequality}
\end{equation}
Indeed, the first inequality follows from
$\Ycal_M\subseteq\Xcal$. For the second, given
$\bsx\in\Xcal$, choose $\bsy\in\Ycal_M$ such that
$d(\bsx,\bsy)\le\varepsilon_M$ and use
\[
d(\bsx,X_n)
\le
d(\bsx,\bsy)+d(\bsy,X_n)
\le
\varepsilon_M+R_n.
\]
Consequently, an $(\alpha_n)$-approximate candidate step satisfies
\begin{equation}
d(\bsx_{n+1},X_n)
\ge
\alpha_n\left(h_{\Xcal}(X_n)-\varepsilon_M\right).
\label{eq:relaxed-continuous-greedy}
\end{equation}
Thus, finite-candidate FPS may be viewed as a relaxed version of the continuous greedy-packing rule, with a relaxation caused by both the candidate discretization error and the inexact farthest-point search. This observation is closely related to the deterministic finite-candidate and relaxed-greedy framework studied in \cite[Section~3.3]{PZ23}.

\section{Exact farthest-point sampling from random candidates}\label{sec:exact}

Let $2\le N\le M$, and let $\Ycal_M=\{Y_1,\ldots,Y_M\}$ be the iid candidate set introduced in \Cref{subsec:random-candidates}. Starting from an arbitrary $\bsx_1\in\Ycal_M$, we apply exact FPS and obtain the nested point sets
\[
X_n=\{\bsx_1,\ldots,\bsx_n\},
\quad 1\le n\le N.
\]
Recall that $R_n=h_{\Ycal_M}(X_n)$ denotes the residual covering radius of the candidate set and $\varepsilon_M=h_{\Xcal}(\Ycal_M)$ denotes the covering error incurred by replacing the continuous domain $\Xcal$ with the finite candidate set.

The analysis consists of two independent ingredients. First, we establish deterministic bounds that hold for every candidate set. Second, we control $\varepsilon_M$ when the candidates are sampled independently from the reference measure $\mu$. Combining the two ingredients gives non-asymptotic probabilistic bounds for the mesh ratio. We then prove that the resulting candidate-pool order $N\log N$ cannot be improved within the iid candidate model, even if FPS is replaced by an arbitrary subset-selection procedure.

\subsection{Deterministic candidate-set estimates}
\label{subsec:deterministic-exact}

We begin with a deterministic result that isolates the effect of the candidate covering error.

\begin{proposition}[Deterministic FPS bounds]
\label{prop:deterministic-exact-fps}
Let $\Ycal_M\subset\Xcal$ be a set of $M$ distinct candidate points, let $\varepsilon_M=h_{\Xcal}(\Ycal_M)$, and let $X_1,\ldots,X_M$ be generated by exact FPS on
$\Ycal_M$. Then, for every $2\le n\le M$,
\begin{equation}
\rho_{\Xcal}(X_n)
\le
2\left(1+\frac{\varepsilon_M}{R_{n-1}}\right).
\label{eq:deterministic-mesh-computable}
\end{equation}
In particular, if $\varepsilon_M<h_{n-1}^*(\Xcal)$, then
\begin{equation}
\rho_{\Xcal}(X_n)
\le
\frac{2}{1-\varepsilon_M/h_{n-1}^*(\Xcal)}.
\label{eq:deterministic-mesh-bound}
\end{equation}
Consequently, if $2\le N\le M$ and $\varepsilon_M<h_{N-1}^*(\Xcal)$, then
\begin{equation}
\max_{2\le n\le N}\rho_{\Xcal}(X_n)
\le
\frac{2}{1-\varepsilon_M/h_{N-1}^*(\Xcal)}.
\label{eq:deterministic-simultaneous-bound}
\end{equation}
\end{proposition}

\begin{proof}
By \eqref{eq:discretization-inequality}, the monotonicity of $R_n$,
and \eqref{eq:exact-separation-identity}, we have
\[
    \rho_{\Xcal}(X_n)
    =
    \frac{h_{\Xcal}(X_n)}{q(X_n)}
    \le
    \frac{R_n+\varepsilon_M}{R_{n-1}/2}
    \le
    2\left(1+\frac{\varepsilon_M}{R_{n-1}}\right),
\]
which proves \eqref{eq:deterministic-mesh-computable}. Moreover,
\[
    R_{n-1}
    \ge
    h_{\Xcal}(X_{n-1})-\varepsilon_M
    \ge
    h_{n-1}^*(\Xcal)-\varepsilon_M.
\]
Hence, if $\varepsilon_M<h_{n-1}^*(\Xcal)$,
\[
    1+\frac{\varepsilon_M}{R_{n-1}}
    \le
    \frac{1}{
        1-\varepsilon_M/h_{n-1}^*(\Xcal)
    },
\]
which proves \eqref{eq:deterministic-mesh-bound}. Finally, $h_{n-1}^*(\Xcal)\ge h_{N-1}^*(\Xcal)$ for $2\le n\le N$, and the simultaneous estimate follows.
\end{proof}
\begin{remark}
\label{rem:relation-pz-finite}
The discretization inequality \eqref{eq:discretization-inequality} and the resulting finite-candidate mesh-ratio estimate are closely related to \cite[Lemma~3.8 and Theorem~3.9]{PZ23}. The formulation in \Cref{prop:deterministic-exact-fps} is convenient for the present probabilistic analysis because the random quantity $\varepsilon_M$ is separated from the deterministic optimal scale $h_{n-1}^*(\Xcal)$. The proof uses only the metric structure and does not require $\Xcal$ to be a subset of Euclidean space.
\end{remark}

\begin{remark}
\label{rem:two-deterministic-errors}
The bound shown in \eqref{eq:deterministic-mesh-computable} has a direct geometric interpretation. The factor $2$ is the mesh-ratio bound for continuous-domain greedy packing, whereas the additional term $2\varepsilon_M/R_{n-1}$ measures the error caused by discretizing $\Xcal$ by the finite candidate set. Thus the candidate covering error must be small relative to the insertion scale $R_{n-1}$.
\end{remark}

\subsection{Non-asymptotic probabilistic bounds}
\label{subsec:probabilistic-upper}

We now control the candidate covering error. We first state a bound that is valid on an arbitrary compact metric-measure space.

\begin{lemma}[Random covering bound]
\label{lem:random-covering-upper}
Let $Y_1,\ldots,Y_M\overset{\mathrm{iid}}{\sim}\mu$ and let
$\varepsilon_M=h_{\Xcal}(\Ycal_M)$. Then, for every $0<r\le D_{\Xcal}$,
\begin{equation}
\PP(\varepsilon_M>r)
\le
\Ncal_{\Xcal}(r/2)
\exp\left(
-M\underline{\mu}(r/2)
\right).
\label{eq:general-covering-tail}
\end{equation}
\end{lemma}

\begin{proof}
Let $\bsz_1,\ldots,\bsz_K$ be the centers of an $r/2$-cover of $\Xcal$, where $K=\Ncal_{\Xcal}(r/2)$. Suppose that every ball $B(\bsz_j,r/2)$ contains at least one candidate point. For any $\bsx\in\Xcal$, choose $\bsz_j$ such that $d(\bsx,\bsz_j)\le r/2$, and then choose $Y_i\in B(\bsz_j,r/2)$. The triangle inequality gives $d(\bsx,Y_i)\le d(\bsx,\bsz_j)+d(\bsz_j,Y_i)\le r$. It follows that $\varepsilon_M\le r$.

Consequently, if $\varepsilon_M>r$, then at least one ball $B(\bsz_j,r/2)$ contains no candidate. By the union bound, we obtain
\begin{align*}
\PP(\varepsilon_M>r)
& \le
\sum_{j=1}^K
\PP\left(
\Ycal_M\cap B(\bsz_j,r/2)=\emptyset
\right)
=
\sum_{j=1}^K
\left(1-\mu(B(\bsz_j,r/2))\right)^M
\\
& \le
K\exp\left(
-M\underline{\mu}(r/2)
\right),
\end{align*}
which is \eqref{eq:general-covering-tail}.
\end{proof}

Combining this lemma with \Cref{prop:deterministic-exact-fps} immediately gives a general probabilistic mesh-ratio estimate.

\begin{theorem}[General probabilistic mesh-ratio bound]
\label{thm:general-probabilistic-upper}
Let $2\le N\le M$, and let $X_1,\ldots,X_N$ be generated by exact FPS from the iid candidate set $\Ycal_M$. Then, for every $0<r<h_{N-1}^*(\Xcal)$,
\begin{equation}
\PP\left(
\max_{2\le n\le N}\rho_{\Xcal}(X_n)
>
\frac{2}{1-r/h_{N-1}^*(\Xcal)}
\right)
\le \Ncal_{\Xcal}(r/2)\exp\left(-M\underline{\mu}(r/2)\right).
\label{eq:general-probabilistic-mesh-bound}
\end{equation}
\end{theorem}

\begin{proof}
On the event $\{\varepsilon_M\le r\}$, \eqref{eq:deterministic-simultaneous-bound} gives
\[
\max_{2\le n\le N}\rho_{\Xcal}(X_n)
\le
\frac{2}{
1-\varepsilon_M/h_{N-1}^*(\Xcal)
}
\le
\frac{2}{
1-r/h_{N-1}^*(\Xcal)
}.
\]
The result therefore follows from \Cref{lem:random-covering-upper}.
\end{proof}

We next specialize the result to an $s$-regular metric-measure space and
make all quantities explicit. For $0<\delta<1$, define
\begin{equation}
L_{M,\delta}
:=
\max\left\{
1,\log\left(\frac{2^sM}{\delta}\right)
\right\}\quad \text{and}\quad 
r_{M,\delta}
:=
2\left(
\frac{L_{M,\delta}}{c_-M}
\right)^{1/s}.
\end{equation}

\begin{theorem}[Non-asymptotic bound on an $s$-regular space]
\label{thm:s-regular-upper}
Suppose that $(\Xcal,d,\mu)$ is $s$-regular. Then
\begin{equation}
\PP\left(
\varepsilon_M\le r_{M,\delta}
\right)
\ge 1-\delta.
\label{eq:candidate-covering-high-probability}
\end{equation}
Furthermore, define
\begin{equation}
\tau_{N,M,\delta}
:=
2\left(
\frac{
c_+(N-1)L_{M,\delta}
}{
c_-M
}
\right)^{1/s}.
\label{eq:def-tau}
\end{equation}
If $\tau_{N,M,\delta}<1$, then the exact FPS construction satisfies
\begin{equation}
\PP\left[
\max_{2\le n\le N}\rho_{\Xcal}(X_n)
\le
\frac{2}{1-\tau_{N,M,\delta}}
\right]
\ge
1-\delta.
\label{eq:main-upper-bound}
\end{equation}
\end{theorem}

\begin{proof}
If $r_{M,\delta}\ge D_{\Xcal}$, then $\varepsilon_M\le D_{\mathcal{X}}\le r_{M,\delta}$ deterministically, and hence \eqref{eq:candidate-covering-high-probability} is immediate. We may therefore assume that $r_{M,\delta}<D_{\Xcal}$. By \Cref{lem:covering-number}, we have
\[
\Ncal_{\Xcal}(r/2)
\le
\frac{4^s}{c_-r^s},
\]
while $s$-regularity gives
\[
\underline{\mu}(r/2)
\ge
c_-\left(\frac r2\right)^s.
\]
Consequently, \eqref{eq:general-covering-tail} implies
\begin{equation}
\PP(\varepsilon_M>r)
\le
\frac{4^s}{c_-r^s}
\exp\left(
-\frac{c_-Mr^s}{2^s}
\right).
\label{eq:s-regular-covering-tail}
\end{equation}
Substituting $r=r_{M,\delta}$ gives
\[
\frac{c_-Mr_{M,\delta}^s}{2^s}
=
L_{M,\delta} \quad \text{and}\quad 
\frac{4^s}{c_-r_{M,\delta}^s}
=
\frac{2^sM}{L_{M,\delta}}.
\]
Therefore,
\[
\PP(\varepsilon_M>r_{M,\delta})
\le
\frac{2^sM}{L_{M,\delta}}
\exp(-L_{M,\delta})
\le
\delta,
\]
which proves \eqref{eq:candidate-covering-high-probability}.

By \eqref{eq:optimal-fill-bounds}, we have $h_{N-1}^*(\Xcal) \ge \left(c_+(N-1)\right)^{-1/s}$. It follows that
\[
\frac{r_{M,\delta}}{h_{N-1}^*(\Xcal)}
\le
2\left(
\frac{
c_+(N-1)L_{M,\delta}
}{
c_-M
}
\right)^{1/s}
=
\tau_{N,M,\delta}.
\]
On the event $\{\varepsilon_M\le r_{M,\delta}\}$, \eqref{eq:deterministic-simultaneous-bound} therefore yields
\[
\max_{2\le n\le N}\rho_{\Xcal}(X_n)
\le
\frac{2}{1-\tau_{N,M,\delta}}.
\]
The probability of this event is at least $1-\delta$.
\end{proof}

\begin{remark}
\label{rem:covering-radius-order}
The upper bound $\varepsilon_M\lesssim\left(M^{-1}\log(M/\delta)\right)^{1/s}$ has the standard iid random-covering order. More precise asymptotic constants are known for spheres, cubes, balls, smooth manifolds, and other specific domains; see \cite{J87,RS16}. The present elementary estimate is sufficient for identifying the sharp candidate-pool order.
\end{remark}

Theorem~\ref{thm:s-regular-upper} can be restated in terms of a prescribed mesh-ratio bound.

\begin{corollary}[Candidate size for a prescribed mesh ratio]
\label{cor:target-mesh-ratio}
Suppose that $(\Xcal,d,\mu)$ is $s$-regular. Let $B>2$ and $0<\delta<1$. If
\begin{equation}
M
\ge
\frac{
2^sc_+
}{
c_-\left(1-2/B\right)^s
}
(N-1)L_{M,\delta},
\label{eq:sufficient-candidate-size}
\end{equation}
then
\begin{equation}
\PP\left[
\max_{2\le n\le N}\rho_{\Xcal}(X_n)
\le B
\right]
\ge
1-\delta.
\label{eq:target-mesh-probability}
\end{equation}
In particular, for every fixed $B>2$ and $\gamma>0$, there exists a
constant $C>0$, depending only on $B,\gamma,s,c_-$, and $c_+$, such that $M=\left\lceil CN\log N\right\rceil$ implies
\[
\PP\left[
\max_{2\le n\le N}\rho_{\Xcal}(X_n)
\le B
\right]
\ge
1-N^{-\gamma}
\]
for all sufficiently large $N$.
\end{corollary}

\begin{proof}
The condition \eqref{eq:sufficient-candidate-size} implies
$\tau_{N,M,\delta} \le 1-2/B$. Therefore, $2/(1-\tau_{N,M,\delta})\le B$ and \eqref{eq:target-mesh-probability} follows from
\Cref{thm:s-regular-upper}.

For $\delta=N^{-\gamma}$ and $M=\lceil CN\log N\rceil$, we have
\[
L_{M,\delta}
=
(1+\gamma+o(1))\log N.
\]
Thus, \eqref{eq:sufficient-candidate-size} holds for all sufficiently large $N$ if $C$ is chosen sufficiently large.
\end{proof}

If the candidate pool grows faster than the critical order, the mesh-ratio upper bound approaches the optimal continuous-greedy constant.

\begin{corollary}[Supercritical candidate pools]
\label{cor:supercritical-upper}
Suppose that $(\Xcal,d,\mu)$ is $s$-regular and that
$M=M_N\ge N$ satisfies
\begin{equation}
\frac{M_N}{N\log M_N}\to \infty
\quad\text{as $N\to\infty$.}
\label{eq:supercritical-condition}
\end{equation}
Then, for every $\eta>0$,
\begin{equation}
\PP\left[
\max_{2\le n\le N}\rho_{\Xcal}(X_n)
\le 2+\eta
\right]
\to 1.
\label{eq:mesh-upper-converges-two}
\end{equation}
\end{corollary}

\begin{proof}
Take $\delta_N=M_N^{-1}$. Then $L_{M_N,\delta_N} = O(\log M_N)$, and hence
\[
\tau_{N,M_N,\delta_N}
=
O\left[
\left(
\frac{N\log M_N}{M_N}
\right)^{1/s}
\right]
\to 0.
\]
Since $\delta_N\to0$, the result follows from \eqref{eq:main-upper-bound}.
\end{proof}

\begin{remark}
\label{rem:upper-bound-not-convergence}
For the point sets constructed separately for each $N$, \Cref{cor:supercritical-upper} asserts that the mesh-ratio upper bound tends to $2$. It does not assert that $\rho_{\Xcal}(X_N)$ itself converges to $2$, since an individual finite point set may have mesh ratio smaller than $2$. In \Cref{sec:nested}, the construction produces one nested infinite sequence; there the universal lower bound of \cite{PZ23} yields an exact limit-superior statement on positive-volume Euclidean domains.
\end{remark}

\subsection{Optimal candidate-pool complexity}
\label{subsec:optimal-candidate-complexity}

The upper bound in \Cref{thm:s-regular-upper} shows that $M$ of order $N\log N$ is sufficient for a bounded mesh ratio. We now prove that this order is also necessary for every method that must select its output from an iid candidate pool. The key observation is that an iid candidate set contains an empty ball whose radius is of order $(\log M/M)^{1/s}$. Every subset of the candidate set inherits this hole.

\begin{lemma}[Empty-ball lower bound]
\label{lem:empty-ball-lower}
Suppose that $(\Xcal,d,\mu)$ is $s$-regular. Let $r>0$ satisfy
$2r\le D_{\Xcal}$ and $c_+r^s\le 1/2$. Then
\begin{equation}
\PP(\varepsilon_M>r)
\ge
1-
c_+2^sr^s
\exp\left(2c_+Mr^s\right).
\label{eq:empty-ball-nonasymptotic}
\end{equation}
\end{lemma}

\begin{proof}
Choose an inclusion-wise maximal set $\{\bsz_1,\ldots,\bsz_K\}\subset\Xcal$ whose pairwise distances are greater than $2r$. By maximality, the balls $B(\bsz_j,2r)$ cover $\Xcal$. The upper regularity bound therefore gives
\[
1
\le
\sum_{j=1}^K\mu(B(\bsz_j,2r))
\le
Kc_+(2r)^s,
\]
and hence
\begin{equation}
K\ge\frac{1}{c_+2^sr^s}.
\label{eq:number-disjoint-balls}
\end{equation}
The balls $B(\bsz_1,r),\ldots,B(\bsz_K,r)$ are pairwise disjoint.

Let
\[
I_j
:=
\mathds{1}_{
\Ycal_M\cap B(\bsz_j,r)=\emptyset
},
\quad
Z:=\sum_{j=1}^K I_j.
\]
Writing $p_j=\mu(B(\bsz_j,r))$, we have
$p_j\le c_+r^s\le1/2$. The elementary inequality
\[
1-u\ge e^{-2u},
\quad 0\le u\le\frac12,
\]
gives
\[
\EE[I_j]
=
(1-p_j)^M
\ge
\exp(-2Mp_j)
\ge
\exp(-2c_+Mr^s).
\]
Together with \eqref{eq:number-disjoint-balls}, this yields
\begin{equation}
\EE[Z]
\ge
\frac{
\exp(-2c_+Mr^s)
}{
c_+2^sr^s
}.
\label{eq:expected-empty-balls}
\end{equation}

For $i\ne j$, the corresponding balls are disjoint, and therefore
\[
\EE[I_iI_j]
=
(1-p_i-p_j)^M
\le
\left((1-p_i)(1-p_j)\right)^M
=
\EE[I_i]\EE[I_j].
\]
Thus the variables $I_1,\ldots,I_K$ are pairwise negatively correlated,
and
\[
\operatorname{Var}(Z)
\le
\sum_{j=1}^K\operatorname{Var}(I_j)
\le
\EE[Z].
\]
Chebyshev's inequality now gives
\[
\PP(Z=0)
\le
\frac{\operatorname{Var}(Z)}{(\EE[Z])^2}
\le
\frac{1}{\EE[Z]}
\le
c_+2^sr^s\exp(2c_+Mr^s).
\]
If $Z>0$, one of the balls $B(\bsz_j,r)$ contains no candidate point, so $d(\bsz_j,\Ycal_M)>r$ and hence $\varepsilon_M>r$. This proves \eqref{eq:empty-ball-nonasymptotic}.
\end{proof}

For $\theta\in(0,1)$, define
\begin{equation}
\underline{r}_{M,\theta}
:=
\left(
\frac{\theta\log M}{2c_+M}
\right)^{1/s}.
\label{eq:def-r-lower}
\end{equation}

\begin{corollary}[Lower bound for the candidate covering radius]
\label{cor:candidate-covering-lower}
Suppose that $(\Xcal,d,\mu)$ is $s$-regular. For every fixed $\theta\in(0,1)$ and all sufficiently large $M$,
\begin{equation}
\PP\left[
\varepsilon_M>
\left(
\frac{\theta\log M}{2c_+M}
\right)^{1/s}
\right]
\ge
1-
2^{s-1}\theta
\frac{\log M}{M^{1-\theta}}.
\label{eq:candidate-covering-lower-probability}
\end{equation}
In particular,
\begin{equation}
\varepsilon_M
\asymp_{\PP}
\left(\frac{\log M}{M}\right)^{1/s},
\label{eq:covering-radius-order-probability}
\end{equation}
where the notation means that the ratio is bounded above and below by
positive constants with probability tending to~$1$.
\end{corollary}

\begin{proof}
For sufficiently large $M$, the radius $r=\underline{r}_{M,\theta}$ satisfies the assumptions of \Cref{lem:empty-ball-lower}. Moreover, we have $2c_+Mr^s=\theta\log M$. Substituting this radius into \eqref{eq:empty-ball-nonasymptotic} gives
\[
c_+2^sr^s
\exp(2c_+Mr^s)
=
c_+2^s
\frac{\theta\log M}{2c_+M}
M^\theta
=
2^{s-1}\theta
\frac{\log M}{M^{1-\theta}},
\]
which proves
\eqref{eq:candidate-covering-lower-probability}. The upper estimate in
\eqref{eq:covering-radius-order-probability} follows from
\Cref{thm:s-regular-upper}.
\end{proof}

We now prove the lower bound uniformly over all possible subsets of the candidate cloud.

\begin{theorem}[Lower bound for arbitrary candidate thinning]
\label{thm:arbitrary-thinning-lower}
Suppose that $(\Xcal,d,\mu)$ is $s$-regular, let $2\le N\le M$, and fix $\theta\in(0,1)$. Then, for all sufficiently large $M$,
\begin{equation}
\PP\left[
\inf_{Z\subseteq\Ycal_M,\ |Z|=N}
\rho_{\Xcal}(Z)
\ge
\left(
\frac{
\theta c_-N\log M
}{
2c_+M
}
\right)^{1/s}
\right]
\ge
1-
2^{s-1}\theta
\frac{\log M}{M^{1-\theta}}.
\label{eq:arbitrary-thinning-lower}
\end{equation}
\end{theorem}

\begin{proof}
Every subset $Z\subseteq\Ycal_M$ with $|Z|=N$ satisfies $h_{\Xcal}(Z) \ge h_{\Xcal}(\Ycal_M) = \varepsilon_M$. On the other hand, \eqref{eq:separation-upper-bound} gives $q(Z)\le(c_-N)^{-1/s}$. Consequently, $\rho_{\Xcal}(Z) \ge \varepsilon_M(c_-N)^{1/s}$. On the event in \eqref{eq:candidate-covering-lower-probability}, this implies
\[
\rho_{\Xcal}(Z)
\ge
\left(
\frac{
\theta c_-N\log M
}{
2c_+M
}
\right)^{1/s}
\]
simultaneously for every $N$-point subset of $\Ycal_M$.
\end{proof}

\begin{corollary}[Necessity of logarithmic oversampling]
\label{cor:necessity-logarithmic-oversampling}
Suppose that $(\Xcal,d,\mu)$ is $s$-regular. Let $N\to\infty$ and let $M=M_N\ge N$ satisfy $M_N=o(N\log N)$. Then
\begin{equation}
\inf_{Z\subseteq\Ycal_{M_N},\,
|Z|=N}
\rho_{\Xcal}(Z)
\to \infty
\quad\text{in probability}.
\label{eq:all-thinning-diverges}
\end{equation}
In particular, the mesh ratio diverges in probability for the output of every deterministic or randomized procedure that is required to select $N$ points from the iid candidate set $\Ycal_{M_N}$.
\end{corollary}

\begin{proof}
Since $M_N\ge N$, we have $\log M_N\ge\log N$. The condition $M_N=o(N\log N)$ implies
\[
\frac{N\log M_N}{M_N}
\ge
\frac{N\log N}{M_N}
\to \infty.
\]
The deterministic lower bound inside the probability in \eqref{eq:arbitrary-thinning-lower} consequently tends to infinity, whereas its failure probability tends to zero for every fixed $\theta\in(0,1)$.
\end{proof}

Combining \Cref{cor:target-mesh-ratio,cor:necessity-logarithmic-oversampling} gives the principal
candidate-complexity conclusion of this section: $M=\Theta(N\log N)$ is the sharp order of the iid candidate-pool size required for constructing an $N$-point set with bounded mesh ratio. A sufficiently large multiple of $N\log N$ allows exact FPS to construct such a set with high probability, whereas $o(N\log N)$ candidates are insufficient for every possible selection rule.

\begin{remark}[The logarithmic factor is specific to iid candidates]
\label{rem:iid-logarithmic-factor}
The lower bound in \Cref{thm:arbitrary-thinning-lower} is not a limitation of FPS. It results from the largest empty regions of an iid candidate cloud and holds for any thinning algorithm. If a different candidate generator satisfies $h_{\Xcal}(\Ycal_M) = O_{\PP}(M^{-1/s})$, a bounded mesh ratio can be obtained with $M=CN$, where $C$ is chosen sufficiently large according to the desired mesh-ratio bound and failure probability.
\end{remark}

\section{Approximate and fast farthest-point sampling}
\label{sec:approximate}

The direct implementation of exact FPS maintains the distance of every candidate point from the current selected set and updates all these distances after every insertion. Its cost is $O(MN)$ distance evaluations. At the critical iid candidate-pool size $M\asymp N\log N$, this becomes $O(N^2\log N)$.

In this section, we first quantify how an inexact farthest-point search affects the geometric quality of the selected set. We then discuss two routes to faster computation. The first combines an existing exact greedy-permutation algorithm for doubling metrics with a probabilistic bound on the spread of the random candidate cloud. The second is a simple multiscale approximate FPS algorithm that does not require a prescribed distance scale and stops automatically after exactly $N$ points have been selected. We subsequently place exact FPS, approximate FPS, and the multiscale construction in a common metric-net framework and explain their relation to single-scale and hierarchical maximal independent sets. Finally, for fixed-dimensional Euclidean spaces, we give a spatial-hashing implementation of the multiscale method and derive its end-to-end complexity bounds. In what follows, for the end-to-end complexity statements, we assume that each independent sample from $\mu$ can be generated in constant time.

\subsection{Mesh-ratio bounds for approximate sampling}
\label{subsec:approximate-bounds}

Let $\Ycal_M\subset\Xcal$ be a deterministic candidate set and let
\[
X_n=\{\bsx_1,\ldots,\bsx_n\},
\quad
1\le n\le N,
\]
be generated by the approximate FPS rule
\[
d(\bsx_{n+1},X_n)
\ge
\alpha_n R_n,
\quad
R_n=h_{\Ycal_M}(X_n),
\]
where $\alpha_n\in(0,1]$. Recall the notation $\underline{\alpha}_n := \min_{1\le j\le n}\alpha_j$.

The following result extends
\Cref{prop:deterministic-exact-fps} to approximate FPS.

\begin{proposition}[Deterministic bounds for approximate FPS]
\label{prop:deterministic-approximate-fps}
Let $\varepsilon_M=h_{\Xcal}(\Ycal_M)$. Then, for every
$2\le n\le N$,
\[
    \rho_{\Xcal}(X_n)
    \le
    \frac{2}{\underline{\alpha}_{n-1}}
    \left(
        1+\frac{\varepsilon_M}{R_{n-1}}
    \right).
\]
If $\varepsilon_M<h_{n-1}^*(\Xcal)$, then
\[
    \rho_{\Xcal}(X_n)
    \le
    \frac{2}{
        \underline{\alpha}_{n-1}
        \left(
            1-\varepsilon_M/h_{n-1}^*(\Xcal)
        \right)
    }.
\]
In particular, if
$\varepsilon_M<h_{N-1}^*(\Xcal)$, then
\begin{equation}
\max_{2\le n\le N}
\rho_{\Xcal}(X_n)
\le
\frac{2}{
\underline{\alpha}_{N-1}
\left(
1-\varepsilon_M/h_{N-1}^*(\Xcal)
\right)
}.
\label{eq:approximate-mesh-simultaneous}
\end{equation}
\end{proposition}

\begin{proof}
By \eqref{eq:approximate-separation-bound},
\[
    q(X_n)
    \ge
    \frac{\underline{\alpha}_{n-1}}{2}R_{n-1}.
\]
On the other hand,
\[
    h_{\Xcal}(X_n)
    \le
    R_n+\varepsilon_M
    \le
    R_{n-1}+\varepsilon_M.
\]
Dividing the latter inequality by the former proves the first
estimate. The remaining claims follow exactly as in
\Cref{prop:deterministic-exact-fps}.
\end{proof}

The two sources of error enter multiplicatively: $1/\underline{\alpha}_{N-1}$ and $1/(1-\varepsilon_M/h_{N-1}^*(\Xcal))$. For small errors, the excess over~$2$ is of first order in $1-\underline{\alpha}_{N-1}$ and
$\varepsilon_M/h_{N-1}^*(\Xcal)$.

Combining \Cref{prop:deterministic-approximate-fps} with the random-covering estimate from \Cref{thm:s-regular-upper} gives the following result.

\begin{theorem}[Probabilistic bound for approximate FPS]
\label{thm:probabilistic-approximate-fps}
Suppose that $(\Xcal,d,\mu)$ is $s$-regular. Let
$2\le N\le M$, let $0<\delta<1$, and let
\[
\tau_{N,M,\delta}
=
2\left(
\frac{
c_+(N-1)L_{M,\delta}
}{
c_-M
}
\right)^{1/s}
\]
be as in \eqref{eq:def-tau}. If
$\tau_{N,M,\delta}<1$, then
\begin{equation}
\PP\left[
\max_{2\le n\le N}\rho_{\Xcal}(X_n)
\le
\frac{2}{
\underline{\alpha}_{N-1}
(1-\tau_{N,M,\delta})
}
\right]
\ge
1-\delta.
\label{eq:probabilistic-approximate-bound}
\end{equation}
\end{theorem}

\begin{proof}
With probability at least $1-\delta$, we have $\varepsilon_M/h_{N-1}^*(\Xcal) \le \tau_{N,M,\delta}$ by the proof of \Cref{thm:s-regular-upper}. The result now follows from \eqref{eq:approximate-mesh-simultaneous}.
\end{proof}

Similarly to \Cref{cor:target-mesh-ratio} for the exact FPS, we can obtain a prescribed mesh-ratio bound by solving the condition in \Cref{thm:probabilistic-approximate-fps} for $M$.

\begin{corollary}[Asymptotically accurate approximate FPS]
\label{cor:asymptotically-accurate-approximate}
Suppose that $(\mathcal{X},d,\mu)$ is $s$-regular and that $M=M_N\ge N$ satisfies $M_N/(N\log M_N)\to \infty$. For each $N$, let the $N$-point construction use approximation factors $\alpha_{1,N},\ldots,\alpha_{N-1,N}\in(0,1]$, and define $\underline{\alpha}_N := \min_{1\le j<N}\alpha_{j,N}$. If $\underline{\alpha}_N\to1$, then, for every $\eta>0$,
\begin{equation}
\PP\left[
\max_{2\le n\le N}\rho_{\Xcal}(X_n)
\le 2+\eta
\right]
\to 1.
\label{eq:approximate-convergence-two}
\end{equation}
More generally, if $\liminf_{N\to\infty}\underline{\alpha}_N\ge\alpha>0$, then
\begin{equation}
\max_{2\le n\le N}\rho_{\mathcal{X}}(X_n)
    \le \frac{2}{\alpha}+o_{\PP}(1).
\label{eq:fixed-approximation-asymptotic}
\end{equation}
\end{corollary}

\begin{proof}
As in the proof of \Cref{cor:supercritical-upper}, taking
$\delta_N=M_N^{-1}$ gives $\tau_{N,M_N,\delta_N}\to 0$ and $\delta_N\to 0$. The claims follow from \eqref{eq:probabilistic-approximate-bound}.
\end{proof}

\subsection{Direct and fast exact implementations}
\label{subsec:direct-fast-exact}

The direct implementation of exact FPS stores its current distance
\[
\Delta_n(\bsy):=d(\bsy,X_n),
\]
for every candidate $\bsy\in\Ycal_M$. After selecting $\bsx_{n+1}$, the distances are updated by $\Delta_{n+1}(\bsy)=\min\{\Delta_n(\bsy),d(\bsy,\bsx_{n+1})\}$. A scan of the candidates identifies a maximizer of $\Delta_n$, and another
scan performs the updates.

\begin{proposition}[Direct implementation]
\label{prop:direct-fps-complexity}
The first $N$ points of an exact farthest-point traversal of an $M$-point candidate set can be computed using $O(MN)$ distance evaluations and comparisons, together with $O(M)$ auxiliary storage.
\end{proposition}

The $O(MN)$ cost is specific to the direct implementation and need not be optimal when additional geometric structure is available. For finite metrics of bounded doubling dimension, significantly faster algorithms are known. We recall the relevant result after establishing that $s$-regular spaces have a uniformly bounded doubling constant.

Let $\lambda_{\Xcal}$ denote the smallest integer such that every ball of radius $2r$ in $\Xcal$ can be covered by $\lambda_{\Xcal}$ balls of radius $r$.

\begin{lemma}
\label{lem:s-regular-doubling-constant}
If $(\Xcal,d,\mu)$ is $s$-regular, then $\lambda_{\Xcal} \le \lceil 5^s\, c_+/c_-\rceil$.
\end{lemma}

\begin{proof}
If $r\ge D_{\Xcal}$, then $B(\bsx,r)=\Xcal$, so the claim is immediate. We may therefore assume that $0<r<D_{\Xcal}$. Fix $\bsx\in\Xcal$ and $r>0$, and take an inclusion-wise maximal subset $\{\bsz_1,\ldots,\bsz_K\}$ of $B(\bsx,2r)$, whose pairwise distances are strictly greater than $r$. Maximality implies that the balls $B(\bsz_j,r)$ cover $B(\bsx,2r)$. The balls $B(\bsz_j,r/2)$ are pairwise disjoint and are contained in $B(\bsx,5r/2)$. Therefore
\[
Kc_-\left(\frac r2\right)^s
\le
\mu(B(\bsx,5r/2))
\le
c_+\left(\frac{5r}{2}\right)^s.
\]
If $5r/2>D_{\Xcal}$, the same upper bound follows from
\[
    \mu(B(\bsx,5r/2))
    =
    1
    \le
    c_+D_{\Xcal}^s
    \le
    c_+(5r/2)^s.
\]
It follows that $K \le 5^s\, c_+/c_-$, which proves the claim.
\end{proof}

Every finite subset $P\subseteq \Xcal$, equipped with the induced metric, has doubling constant at most $\lambda_{\Xcal}^2$. Indeed, an ambient ball of radius $2r$ can be covered by at most $\lambda_{\Xcal}^2$ balls of radius $r/2$. For every such ball having a nonempty intersection with $P$, choose one point of the intersection as a new center. The corresponding balls of radius $r$, now centered in $P$, cover the original ball intersected with $P$. 


For a finite set $P\subset\Xcal$ of at least two distinct points,
define its spread by
\begin{equation}
\Phi(P)
:=
\frac{
\operatorname{diam}(P)
}{
\min_{\substack{\bsx,\bsy\in P\\ \bsx\ne \bsy}}d(\bsx,\bsy)
}.
\label{eq:spread-definition}
\end{equation}

The following result is due to Har-Peled and Mendel \cite[Theorem~3.2]{HM06}.

\begin{proposition}[Fast exact greedy permutation]
\label{prop:har-peled-mendel}
Let $P$ be a finite metric space with $|P|=M$, doubling constant $\lambda$, and spread $\Phi(P)$. Assuming constant-time distance evaluations and the unit-cost floating-point word RAM model used in \cite{HM06}, an exact greedy permutation of $P$ can be computed in  $O\left(\lambda^{O(1)}M\log\left(\Phi(P)M\right)\right)$ time and $O\left(\lambda^{O(1)}M\right)$ space.
\end{proposition}

For the random candidate cloud, the spread is polynomially bounded with
high probability.

\begin{lemma}[Spread of the random candidate set]
\label{lem:random-candidate-spread}
Suppose that $(\Xcal,d,\mu)$ is $s$-regular. Then, for every $u>0$,
\begin{equation}
\begin{aligned}
\PP\left[
\min_{1\le i<j\le M}
d(Y_i,Y_j)
\le u
\right]
\le
\binom{M}{2}c_+u^s.
\end{aligned}
\label{eq:minimum-distance-tail}
\end{equation}
Consequently, for $0<\delta<1$, with probability at least $1-\delta$,
\begin{equation}
\Phi(\Ycal_M)
\le
D_{\Xcal}
\left(
\frac{
c_+M(M-1)
}{
2\delta
}
\right)^{1/s}.
\label{eq:random-spread-bound}
\end{equation}
\end{lemma}

\begin{proof}
For $u>D_{\Xcal}$, the claim is trivial, since $c_+u^s\ge c_+D_{\Xcal}^s\ge1$. Thus we may assume $0<u\le D_{\Xcal}$. For every pair $i<j$,
\[
\PP\bigl(d(Y_i,Y_j)\le u\bigr)
=
\int_{\Xcal}
\mu(B(\bsx,u))
\,\mu(\mathrm{d} x)
\le
c_+u^s.
\]
The union bound proves \eqref{eq:minimum-distance-tail}. Taking
\[
u
=
\left(
\frac{
2\delta
}{
c_+M(M-1)
}
\right)^{1/s}
\]
shows that, with probability at least $1-\delta$, the minimum pairwise distance exceeds $u$. Since $\operatorname{diam}(\Ycal_M)\le D_{\Xcal}$, the spread bound follows.
\end{proof}

\begin{corollary}[Fast exact FPS for random candidates]
\label{cor:fast-exact-random}
Suppose that $(\Xcal,d,\mu)$ is $s$-regular. There exists an exact implementation of FPS such that, with probability at least $1-\delta$, the complete greedy permutation of $\Ycal_M$ can be computed in $O_{s,c_-,c_+,\Xcal}\left(M\log (M/\delta)\right)$ time and $O_{s,c_-,c_+}(M)$ space.
\end{corollary}

\begin{proof}
By \Cref{lem:s-regular-doubling-constant}, the doubling constant is bounded in terms of $s, c_-$ and $c_+$. On the event in \eqref{eq:random-spread-bound}, $\log\left(\Phi(\Ycal_M)M\right) = O_{\Xcal,s,c_+} \left( \log (M/\delta) \right)$. The result follows from \Cref{prop:har-peled-mendel}.
\end{proof}

Combining this result with the mesh-ratio estimate gives an end-to-end high-probability statement for exact FPS.

\begin{corollary}[End-to-end complexity of exact FPS]
\label{cor:end-to-end-exact}
Suppose that $(\Xcal,d,\mu)$ is $s$-regular. Let $B>2$ and $0<\delta<1$. If
\begin{equation}
M
\ge
\frac{
2^sc_+
}{
c_-\left(1-2/B\right)^s
}
(N-1)L_{M,\delta/2},
\label{eq:fast-exact-candidate-condition}
\end{equation}
then, with probability at least $1-\delta$,
\begin{equation}
\max_{2\le n\le N}
\rho_{\Xcal}(X_n)
\le B,
\label{eq:fast-exact-mesh-target}
\end{equation}
and the selected points can be computed in $O_{s,c_-,c_+,\Xcal}\left(
M\log (M/\delta)\right)$ time.

In particular, for fixed $B>2$, fixed regularity constants, and $\delta=N^{-\gamma}$ with $\gamma>0$, a sufficiently large choice $M=\left\lceil CN\log N\right\rceil$ gives both
\[
\max_{2\le n\le N}\rho_{\Xcal}(X_n)\le B
\]
and computational cost $O(N\log^2N)$ with probability tending to~$1$.
\end{corollary}

\begin{proof}
The result follows immediately by first applying \Cref{cor:target-mesh-ratio} with failure probability $\delta/2$ and \Cref{cor:fast-exact-random} with failure probability $\delta/2$, and then using the union bound.
\end{proof}

\begin{remark}
\label{rem:exact-fast-not-simple}
The fast exact algorithm in \Cref{prop:har-peled-mendel} maintains a finite Voronoi partition, local neighbor lists, and a priority queue. Although its asymptotic running time is smaller, its implementation is more elaborate than the direct distance-update method. \Cref{alg:multiscale-fps} offers a simpler alternative with an explicit accuracy--cost trade-off.
\end{remark}

\subsection{A multiscale approximate FPS algorithm}
\label{subsec:multiscale-fps}

We now introduce an approximate FPS algorithm that avoids identifying the
global farthest candidate at every iteration. Fix a parameter $\beta>1$.
Starting from $\bsx_1$, let
\[
r_0
:=
h_{\Ycal_M}(\{\bsx_1\})
=
\max_{\bsy\in\Ycal_M}d(\bsy,\bsx_1),
\]
and define geometrically-decreasing thresholds
\[
r_k:=\beta^{-k}r_0,
\quad k\ge1.
\]
At level $k$, the candidate set is scanned once. A candidate is selected whenever its distance from the currently selected set exceeds $r_k$. The algorithm stops immediately when $N$ points have been selected. \Cref{alg:multiscale-fps} describes the workflow of the multiscale approximate FPS.

\begin{algorithm}[t]
\caption{Multiscale approximate farthest-point sampling}
\label{alg:multiscale-fps}
\begin{algorithmic}[1]
\REQUIRE Candidate set $\Ycal_M$, output size $N\le M$, and
parameter $\beta>1$.
\ENSURE An ordered set $X_N\subseteq\Ycal_M$ with $|X_N|=N$.
\STATE Choose $\bsx_1\in\Ycal_M$ and set
$X\leftarrow \{\bsx_1\}$.
\STATE Set
$r_0\leftarrow\max_{\bsy\in\Ycal_M}d(\bsy,\bsx_1)$ and $k\leftarrow0$.
\WHILE{$|X|<N$}
\STATE Set $k\leftarrow k+1$ and $r_k\leftarrow r_{k-1}/\beta$.
\STATE Initialize or rebuild a threshold-neighbor structure for
$X$ at radius $r_k$.
\FOR{each $\bsy\in\Ycal_M\setminus X$ in a fixed order}
\IF{$d(\bsy,X)>r_k$}
\STATE Append $\bsy$ to the ordered set and set
$X\leftarrow X\cup \{\bsy\}$.
\STATE Insert $\bsy$ into the threshold-neighbor structure.
\IF{$|X|=N$}
\RETURN $X$.
\ENDIF
\ENDIF
\ENDFOR
\ENDWHILE
\end{algorithmic}
\end{algorithm}

Although the algorithm uses a sequence of internal thresholds, the user does not need to prescribe a geometric scale. The only additional input, besides the candidate set and $N$, is the approximation parameter $\beta$.

\begin{proposition}[Approximation guarantee and number of levels]
\label{prop:multiscale-guarantee}
\Cref{alg:multiscale-fps} terminates and its output order is a $\beta^{-1}$-approximate farthest-point traversal. More precisely, every selected point satisfies
\begin{equation}
d(\bsx_{n+1},X_n)
>
\frac{1}{\beta}
h_{\Ycal_M}(X_n).
\label{eq:multiscale-approximation}
\end{equation}

Let $K$ be the level at which the $N$th point is selected. If $\varepsilon_M<h_{N-1}^*(\Xcal)$, then
\begin{equation}
K
\le
1+
\left\lceil
\log_{\beta}
\left(
\frac{
R_1
}{
h_{N-1}^*(\Xcal)-\varepsilon_M
}
\right)
\right\rceil.
\label{eq:number-multiscale-levels}
\end{equation}
\end{proposition}

\begin{proof}
At the beginning of the first level, $h_{\Ycal_M}(X_1)=r_0$. After a full scan at level $k$, every candidate is within distance $r_k$ of the selected set. Indeed, a candidate that is not selected has distance at most $r_k$ when it is inspected, and this distance cannot increase as additional points are selected. Thus, at the beginning of level $k$, we have $h_{\Ycal_M}(X) \le r_{k-1}$. The residual covering radius can only decrease during the scan. Every point selected at level $k$ has insertion distance greater than
\[
r_k
=
\frac{r_{k-1}}{\beta}
\ge
\frac{1}{\beta}
h_{\Ycal_M}(X).
\]
This proves \eqref{eq:multiscale-approximation}.

Since the candidate set is finite and consists of distinct points, its minimum pairwise distance is positive. Once $r_k$ is smaller than this minimum distance, every remaining candidate is selected during the scan. Thus, the algorithm terminates.

Suppose that the $N$th point is selected at level $K$. If $K\ge2$, let
$X_m$ be the selected set after the completed scan at level $K-1$. Then $m<N$ and $h_{\Ycal_M}(X_m)\le r_{K-1}$. By the discretization inequality, $h_{\Xcal}(X_m) \le r_{K-1}+\varepsilon_M$. On the other hand, $h_{\Xcal}(X_m) \ge h_m^*(\Xcal) \ge h_{N-1}^*(\Xcal)$, so $r_{K-1} \ge h_{N-1}^*(\Xcal)-\varepsilon_M$. Since $r_{K-1} = \beta^{-(K-1)}R_1$, the desired estimate follows. When $K=1$, the estimate is immediate from $R_1 \ge h_{\Xcal}(X_1)-\varepsilon_M \ge h_{N-1}^*(\Xcal)-\varepsilon_M$.
\end{proof}

Applying \Cref{prop:deterministic-approximate-fps} with $\alpha=\beta^{-1}$ gives the following immediate consequence.

\begin{corollary}
\label{cor:multiscale-mesh-bound}
If $\varepsilon_M<h_{N-1}^*(\Xcal)$, then the output of \Cref{alg:multiscale-fps} satisfies
\begin{equation}
\max_{2\le n\le N}
\rho_{\Xcal}(X_n)
\le
\frac{
2\beta
}{
1-\varepsilon_M/h_{N-1}^*(\Xcal)
}.
\label{eq:multiscale-mesh-bound}
\end{equation}
If $(\Xcal,d,\mu)$ is $s$-regular and
$\tau_{N,M,\delta}<1$, then
\begin{equation}
\PP\left[
\max_{2\le n\le N}
\rho_{\Xcal}(X_n)
\le
\frac{2\beta}{
1-\tau_{N,M,\delta}
}
\right]
\ge
1-\delta.
\label{eq:multiscale-probabilistic-bound}
\end{equation}
\end{corollary}

\subsection{Metric-net interpretation and maximal independent sets}\label{subsec:metric-net-mis}

The preceding constructions admit a common interpretation as packing--covering certificates on the candidate set. This viewpoint also clarifies their relation to maximal independent sets. Let $\Ycal\subseteq\Xcal$ be finite, write $\varepsilon(\Ycal):=h_{\Xcal}(\Ycal)$, and, for $S\subseteq\Ycal$ with $|S|\ge2$, set 
\[ \operatorname{sep}(S):=2q(S). \]
Given $r>0$, $a>0$, and $b\ge0$, we say that $S$ has an \emph{$(a,b;r)$-net certificate} in $\Ycal$ if 
\[ h_{\Ycal}(S)\le br \quad\text{and}\quad \operatorname{sep}(S)\ge ar. \]

\begin{proposition}[Candidate-to-domain lifting] 
\label{prop:candidate-net-lifting} 
If $S\subseteq\Ycal\subseteq\Xcal$ has an $(a,b;r)$-net certificate, then \begin{equation} 
h_{\Xcal}(S)\le br+\varepsilon(\Ycal), \quad q(S)\ge\frac{ar}{2}, \quad \rho_{\Xcal}(S) \le \frac{2}{a} \left( b+\frac{\varepsilon(\Ycal)}{r} \right). 
\label{eq:candidate-net-lifting} 
\end{equation} 
\end{proposition} 

\begin{proof} 
For any $\bsx\in\Xcal$, choose $\bsy\in\Ycal$ with $d(\bsx,\bsy)\le\varepsilon(\Ycal)$ and then $\bsz\in S$ with $d(\bsy,\bsz)\le br$. The triangle inequality gives $d(\bsx,S)\le\varepsilon(\Ycal)+br$. Taking the supremum over $\bsx$, and then using $q(S)=\operatorname{sep}(S)/2$, proves the result. 
\end{proof} 

For a traversal of $\Ycal$, write $R_n:=h_{\Ycal}(X_n)$. An exact FPS prefix $X_n$ has the certificate 
\[ \left( 1,\frac{R_n}{R_{n-1}};R_{n-1} \right), \] 
whereas an approximate FPS prefix has 
\[ \left( \underline{\alpha}_{n-1}, \frac{R_n}{R_{n-1}};R_{n-1} \right). \] 
Indeed, both have candidate covering radius $R_n$, while their separations are $R_{n-1}$ and at least $\underline{\alpha}_{n-1}R_{n-1}$, respectively. Since $R_n\le R_{n-1}$, \eqref{eq:candidate-net-lifting} recovers the deterministic bounds in \Cref{sec:exact,subsec:approximate-bounds}. Thus the difference between exact and approximate FPS lies entirely in the separation part of the certificate. For $r>0$, define the closed $r$-neighborhood graph $G_r(\Ycal)=(\Ycal,E_r)$ by 
\[ \{\bsx,\bsy\}\in E_r \quad \iff \quad \bsx\ne\bsy \ \text{ and }\ d(\bsx,\bsy)\le r. \] 
When $\Ycal=\Ycal_M$ is an iid candidate cloud, $G_r(\Ycal_M)$ is the random geometric graph associated with $(\Xcal,d,\mu)$. Here a maximal independent set means an inclusion-wise maximal independent set, not necessarily one of maximum cardinality. 

\begin{corollary}[Single-scale maximal independent sets] 
\label{cor:single-scale-mis} 
Let $S_r$ be a maximal independent set of $G_r(\Ycal)$ with $|S_r|\ge2$. Then $S_r$ has a $(1,1;r)$-net certificate and 
\begin{equation} 
\rho_{\Xcal}(S_r) \le 2\left( 1+\frac{\varepsilon(\Ycal)}{r} \right). 
\label{eq:single-scale-mis-bound} 
\end{equation} 
\end{corollary} 

\begin{proof} 
Independence gives $\operatorname{sep}(S_r)>r$. If some $\bsy\in\Ycal\setminus S_r$ satisfied $d(\bsy,S_r)>r$, it could be added while preserving independence, contradicting maximality. Hence $h_{\Ycal}(S_r)\le r$, and \Cref{prop:candidate-net-lifting} applies. 
\end{proof} 

Such maximal separated sets are the usual metric nets; see, for example, \cite{HM06}. Exact FPS generates maximal independent sets without requiring the scale as an input. Fix $r>0$ and let $k$ be the first index such that $R_k\le r$. If $k\ge2$, then $R_{k-1}>r$, so $X_k$ is independent in $G_r(\Ycal)$, while $R_k\le r$ implies maximality; the case $k=1$ is immediate. Thus FPS is a cardinality-controlled nested net construction, whereas a single-scale MIS is scale-controlled. The latter's cardinality may depend both on the candidate geometry and on the particular maximal extension. 

The multiscale method in \Cref{alg:multiscale-fps} provides a direct bridge between the two viewpoints. Let $S_k$ denote the selected set after a complete scan at level $k$, with $S_0=\{\bsx_1\}$. Then $S_k$ is a maximal independent set of $G_{r_k}(\Ycal_M)$: the selection rule preserves independence at scale $r_k$, and a completed scan gives $h_{\Ycal_M}(S_k)\le r_k$. If the algorithm terminates with $N\ge2$ points during level $K$, its output contains $S_{K-1}$ and remains independent at scale $r_K$. Therefore, 
\begin{equation} 
h_{\Ycal_M}(X_N) \le r_{K-1}, \quad \operatorname{sep}(X_N) >r_K=\frac{r_{K-1}}{\beta}, \quad \rho_{\Xcal}(X_N) \le 2\beta \left( 1+\frac{\varepsilon_M}{r_{K-1}} \right). 
\label{eq:hierarchical-mis-certificate} 
\end{equation} 
Thus the completed levels are single-scale metric nets, while an intermediate output has a $(\beta^{-1},1;r_{K-1})$-net certificate. Combining the lower bound on $r_{K-1}$ established in the proof of \Cref{prop:multiscale-guarantee} with \eqref{eq:hierarchical-mis-certificate} recovers \Cref{cor:multiscale-mesh-bound}. In this precise sense, \Cref{alg:multiscale-fps} is also a hierarchical MIS construction. 

Finally, on an $s$-regular space, the scale of a single-scale MIS controls its output cardinality up to constants. 
\begin{lemma}[Scale--cardinality relation for MIS nets] 
\label{lem:mis-scale-cardinality} 
Suppose that $(\Xcal,d,\mu)$ is $s$-regular, let $S_r$ be a maximal independent set of $G_r(\Ycal)$, and assume that 
\[ \varepsilon(\Ycal)\le\eta r \quad\text{and}\quad (1+\eta)r\le D_{\Xcal} \] 
for some $\eta\ge0$. Then 
\begin{equation} 
\frac{1}{c_+(1+\eta)^s r^s} \le |S_r| \le \frac{2^s}{c_-r^s}. 
\label{eq:mis-scale-cardinality} 
\end{equation} 
Consequently, if $\Ycal=\Ycal_M$ is an iid candidate cloud, 
\[ r_M\asymp \left( \frac{\log M}{M} \right)^{1/s}, \quad \PP(\varepsilon_M\le\eta r_M)\to 1, \] 
then we have 
\[ |S_{r_M}| \asymp_{\PP} \frac{M}{\log M}. \] 
Hence an MIS output of order $N$ corresponds, at the level of orders, to the same candidate budget $M\asymp N\log N$ as in the fixed-cardinality FPS construction. \end{lemma} 

\begin{proof} 
The balls $B(\bsx,r/2)$, $\bsx\in S_r$, are pairwise disjoint, and hence $|S_r|c_-(r/2)^s\le 1$. On the other hand, \Cref{prop:candidate-net-lifting} gives $h_{\Xcal}(S_r) \le r+\varepsilon(\Ycal) \le (1+\eta)r$. The balls $B(\bsx,(1+\eta)r)$, $\bsx\in S_r$, therefore cover $\Xcal$, so $1\le |S_r|c_+(1+\eta)^sr^s$. These inequalities prove \eqref{eq:mis-scale-cardinality}, and the asymptotic statement follows by substituting the assumed order of $r_M$. \end{proof} The two parametrizations are therefore complementary. FPS takes the desired cardinality $N$ as input and determines the geometric scale adaptively. A single-scale MIS takes the resolution $r$ as input and returns a cardinality controlled only up to constants. Hierarchical MIS reconciles the two viewpoints: stopping between two completed scales gives an arbitrary prescribed cardinality and is precisely the metric-net mechanism underlying the multiscale approximate-FPS construction. We retain FPS as the primary formulation because the problem studied in this paper prescribes the number of output points.

\subsection{Euclidean implementation and end-to-end complexity}
\label{subsec:end-to-end-complexity}

We next show that \Cref{alg:multiscale-fps} has a particularly simple implementation when $\Xcal\subset\RR^d$ is equipped with the Euclidean metric. At level $k$, all selected points are more than $r_k$ apart. Indeed, every point selected at an earlier level was inserted at a threshold at least $r_k$, while every point selected during the current level has distance greater than $r_k$ from all previously selected points.

Partition $\RR^d$ into axis-parallel cells of side length $\ell_k:=r_k/\sqrt d$. Since every such cell has diameter $r_k$, each cell contains at most one selected point. To determine whether $d(\bsy,X)\le r_k$, it suffices to inspect the cells whose index differs from the cell containing $\bsy$ by at most $\lceil\sqrt d\rceil+1$ in each coordinate. Thus at most $C_d := \left(2\lceil\sqrt d\rceil+3\right)^d$ selected points need to be checked.

The occupied cells can be stored in a hash table whose keys are integer cell indices. When the threshold is decreased, the hash table is rebuilt using the already selected points.

\begin{proposition}[Grid implementation]
\label{prop:grid-implementation}
Suppose that $\Xcal\subset\RR^d$ and that Euclidean distances are used. If \Cref{alg:multiscale-fps} terminates at level $K$, then it can be implemented using $O(C_dMK)$ distance evaluations and expected arithmetic time, under standard constant-expected-time hashing. Its storage cost is $O(M+N)$.
\end{proposition}

\begin{proof}
At each level, every candidate is inspected once. The threshold-neighbor query requires at most $C_d$ distance evaluations because each relevant grid cell contains at most one selected point. Rebuilding the grid takes $O(N)$ operations per level, which is absorbed by $O(M)$ since $N\le M$. The result follows by summing over the $K$ levels.
\end{proof}

For $t>0$, write $\log_{\beta}^+(t):=\max\left\{0,\log t/\log\beta\right\}$.

\begin{theorem}[End-to-end complexity of multiscale FPS]
\label{thm:multiscale-end-to-end}
Suppose that $(\Xcal,d,\mu)$ is an $s$-regular metric-measure space with $\Xcal\subset\RR^d$ and the Euclidean metric. Let $\beta>1$, $B>2\beta$, and $0<\delta<1$. If
\begin{equation}
M
\ge
\frac{
2^sc_+
}{
c_-\left(1-2\beta/B\right)^s
}
(N-1)L_{M,\delta},
\label{eq:multiscale-candidate-condition}
\end{equation}
then, with probability at least $1-\delta$,
\begin{equation}
\max_{2\le n\le N}
\rho_{\Xcal}(X_n)
\le B.
\label{eq:multiscale-target-mesh}
\end{equation}
On the same event, the number of levels satisfies
\begin{equation}
K
\le
1+
\left\lceil
\log_{\beta}^+
\left(
\frac{
BD_{\Xcal}
\bigl(c_+(N-1)\bigr)^{1/s}
}{
2\beta
}
\right)
\right\rceil,
\label{eq:multiscale-level-probabilistic}
\end{equation}
and, consequently, the grid implementation has cost $O(C_dMK)$.
\end{theorem}

\begin{proof}
The condition \eqref{eq:multiscale-candidate-condition} implies $\tau_{N,M,\delta} \le 1-2\beta/B$. Therefore, on the event of probability at least $1-\delta$ from \Cref{thm:s-regular-upper},
\[
\max_{2\le n\le N}
\rho_{\Xcal}(X_n)
\le
\frac{2\beta}{
1-\tau_{N,M,\delta}
}
\le B.
\]
On the same event,
\[
h_{N-1}^*(\Xcal)-\varepsilon_M
\ge
\left(
1-\tau_{N,M,\delta}
\right)
h_{N-1}^*(\Xcal)
\ge
\frac{2\beta}{B}
\bigl(c_+(N-1)\bigr)^{-1/s}.
\]
Since $R_1\le D_{\Xcal}$, the level bound follows from \eqref{eq:number-multiscale-levels}. The complexity estimate then follows from \Cref{prop:grid-implementation}.
\end{proof}

For a fixed approximation parameter and fixed dimension, the number of
levels is logarithmic in $N$. Hence, when $M\asymp N\log N$, the multiscale method has cost $O(N\log^2N)$. The following result makes explicit the dependence on the desired distance from the limiting constant $2$.

\begin{corollary}[Mesh ratio close to $2$]
\label{cor:multiscale-close-two}
Let $0<\eta\le1$, and set $\beta:=1+\eta/3$ and $B:=2(1+\eta)$. Suppose that the assumptions of \Cref{thm:multiscale-end-to-end} hold.
If
\begin{equation}
M
\ge
6^s
\frac{c_+}{c_-}
\eta^{-s}
(N-1)L_{M,\delta},
\label{eq:near-two-candidate-size}
\end{equation}
then
\begin{equation}
\PP\left[
\max_{2\le n\le N}
\rho_{\Xcal}(X_n)
\le
2(1+\eta)
\right]
\ge
1-\delta.
\label{eq:near-two-mesh-bound}
\end{equation}
On the same event, the grid implementation has cost $O\left( C_d\eta^{-1}M\log N \right)$. Consequently, choosing $M$ at the order prescribed by \eqref{eq:near-two-candidate-size} gives $O\left( C_d \eta^{-(s+1)} N L_{M,\delta} \log N \right)$ time.
\end{corollary}

\begin{proof}
The condition \eqref{eq:near-two-candidate-size} gives $\tau_{N,M,\delta} \le \eta/3$. For $0<\eta\le1$, it holds that 
\[
\frac{
1+\eta/3
}{
1-\eta/3
}
\le
1+\eta.
\]
Therefore, we have
\[
\frac{2\beta}{
1-\tau_{N,M,\delta}
}
\le
2(1+\eta).
\]
This proves the mesh-ratio bound. Moreover, since $\log\left(1+\eta/3\right) \gtrsim\eta$, while the logarithmic numerator in \eqref{eq:multiscale-level-probabilistic} is $O(\log N)$ for a fixed domain, we obtain $K=O(\eta^{-1}\log N)$. The complexity estimates follow from \Cref{prop:grid-implementation}.
\end{proof}

\begin{remark}[Relation to known greedy-permutation algorithms]
\label{rem:known-greedy-permutation-algorithms}
\Cref{alg:multiscale-fps} is intended as a simple implementation whose mesh-ratio and computational guarantees can be analyzed directly. It is closely related to standard multiscale net constructions. More sophisticated greedy-permutation algorithms are available.

Har-Peled and Mendel \cite{HM06} construct hierarchical nets and near-exact greedy orderings in expected near-linear time in bounded doubling dimension. A recent deterministic construction of \cite{CSS24} computes a $(1+1/M)$-approximate greedy permutation in $2^{O(d_{\mathrm{dbl}})}M\log M$ time, where $d_{\mathrm{dbl}}$ denotes the doubling dimension. For graph metrics and high-dimensional Euclidean point sets, randomized approximate greedy-permutation algorithms with different complexity bounds are given in \cite{EHS20}. Any such method can be combined with \Cref{thm:probabilistic-approximate-fps}, provided that its approximation guarantee is translated into the factors $\alpha_n$ used here.
\end{remark}

\section{Construction of a quasi-uniform nested sequence}
\label{sec:nested}

The constructions studied in \Cref{sec:exact,sec:approximate} start from a candidate set of prescribed size $M$ and return an $N$-point subset. If this procedure is repeated independently for different values of $N$, the resulting point sets need not be nested. In this section, we instead use a single infinite stream of iid candidates and an increasing candidate budget. This produces one infinite sequence $\bsx_1,\bsx_2,\ldots$ whose initial segments $X_n=\{\bsx_1,\ldots,\bsx_n\}$ are quasi-uniform almost surely.

The analysis differs slightly from that of \Cref{sec:exact}. As the candidate set changes with $n$, the insertion radii need not be nonincreasing. We thus interpret each candidate-based step as a relaxed continuous-domain greedy step and apply a deterministic relaxed-greedy argument.

\subsection{A single-stream randomized construction}
\label{subsec:single-stream}

Let $Y_1,Y_2,\ldots\overset{\mathrm{iid}}{\sim}\mu$ be a single infinite stream of candidate points. Let $0=M_0< M_1\le M_2\le\cdots$ be a nondecreasing sequence of positive integers satisfying
\begin{equation}
M_n\ge n+1,
\quad n\ge1,
\label{eq:nested-budget-basic}
\end{equation}
and define the candidate pool available at iteration $n$ by
\begin{equation}
\Ycal^{(n)}
:=
\Ycal_{M_n}
=
\{Y_1,\ldots,Y_{M_n}\}.
\label{eq:iteration-candidate-pool}
\end{equation}
The monotonicity of $M_n$ ensures that every previously selected point remains in all subsequent candidate pools.

Let $\alpha_n\in(0,1]$. Starting with $\bsx_1=Y_1$, we recursively choose $\bsx_{n+1}\in\Ycal^{(n)}$ such that
\begin{equation}
d(\bsx_{n+1},X_n)
\ge
\alpha_n
h_{\Ycal^{(n)}}(X_n).
\label{eq:nested-approximate-fps}
\end{equation}
Exact FPS corresponds to $\alpha_n=1$ for every $n$. \Cref{alg:nested-randomized-fps} describes the workflow of the FPS for constructing an infinite sequence $\bsx_1,\bsx_2,\ldots$.

\begin{algorithm}[t]
\caption{Nested randomized farthest-point sampling}
\label{alg:nested-randomized-fps}
\begin{algorithmic}[1]
\REQUIRE An iid stream $(Y_j)_{j\ge1}$, nondecreasing candidate budgets
$(M_n)_{n\ge1}$ satisfying $M_n\ge n+1$, and approximation factors
$(\alpha_n)_{n\ge1}\subset(0,1]$.
\ENSURE A nested sequence $X_1\subset X_2\subset\cdots$.
\STATE Set $\bsx_1\leftarrow Y_1$, $X_1\leftarrow\{\bsx_1\}$ and $\Ycal^{(0)}\leftarrow \emptyset$.
\FOR{$n=1,2,\ldots$}
\STATE Set
$\Ycal^{(n)}\leftarrow \Ycal^{(n-1)}\cup \{Y_{M_{n-1}+1},\ldots,Y_{M_n}\}$.
\STATE Choose
$\bsx_{n+1}\in\Ycal^{(n)}\setminus X_n$ such that
$d(\bsx_{n+1},X_n)
\ge
\alpha_n h_{\Ycal^{(n)}}(X_n)$.
\STATE Set $X_{n+1}\leftarrow X_n\cup \{\bsx_{n+1}\}$.
\ENDFOR
\end{algorithmic}
\end{algorithm}

Since $\mu$ is non-atomic, the candidate points are pairwise distinct with probability ~$1$. On this event, \eqref{eq:nested-budget-basic} implies that $\Ycal^{(n)}\setminus X_n$ is nonempty. Moreover, $h_{\Ycal^{(n)}}(X_n)>0$, so \eqref{eq:nested-approximate-fps} always selects a new point. Define the covering error of the candidate pool available at step $n$ by
\begin{equation}
\varepsilon_n
:=
h_{\Xcal}(\Ycal^{(n)})
=
h_{\Xcal}(\Ycal_{M_n}).
\label{eq:nested-candidate-error}
\end{equation}
The discretization inequality \eqref{eq:discretization-inequality} gives $h_{\Ycal^{(n)}}(X_n) \ge h_{\Xcal}(X_n)-\varepsilon_n$. Consequently,
\begin{equation}
d(\bsx_{n+1},X_n)
\ge
\alpha_n
\left(
h_{\Xcal}(X_n)-\varepsilon_n
\right)
=
\alpha_n
\left(
1-\frac{\varepsilon_n}{h_{\Xcal}(X_n)}
\right)
h_{\Xcal}(X_n).
\label{eq:nested-relaxed-greedy}
\end{equation}
Thus the quality of the continuous-domain greedy step is controlled by both the FPS approximation factor $\alpha_n$ and the relative candidate covering error.

We first recall the deterministic principle needed below.

\begin{proposition}[Asymptotic relaxed-greedy criterion]
\label{prop:asymptotic-relaxed-greedy}
Suppose that $(\Xcal,d,\mu)$ is $s$-regular and that $(\bsx_n)_{n\ge1}$ is a sequence of distinct points in $\Xcal$. Let $X_n=\{\bsx_1,\ldots,\bsx_n\}$ and define the effective continuous-domain relaxation factor by
\begin{equation}
\gamma_n
:=
\frac{
d(\bsx_{n+1},X_n)
}{
h_{\Xcal}(X_n)
},
\quad n\ge1.
\label{eq:effective-relaxation-factor}
\end{equation}
Then $\gamma_n\in(0,1]$. If $\underline{\gamma}:= \inf_{n\ge1}\gamma_n>0$, then
\begin{equation}
q(X_n)
\ge
\frac{\underline{\gamma}}{2}
h_{\Xcal}(X_{n-1}) \quad \text{and}\quad \rho_{\Xcal}(X_n)
\le
\frac{2}{\underline{\gamma}},
\quad n\ge2.
\label{eq:uniform-relaxed-separation}
\end{equation}
More generally, if $\gamma_\infty := \liminf_{n\to\infty}\gamma_n>0$, then $(\bsx_n)_{n\ge1}$ is quasi-uniform and
\begin{equation}
\limsup_{n\to\infty}
\rho_{\Xcal}(X_n)
\le
\frac{2}{\gamma_\infty}.
\label{eq:asymptotic-relaxed-mesh}
\end{equation}
\end{proposition}

\begin{proof}
Put $h_n:=h_{\Xcal}(X_n)$ and $q_n:=q(X_n)$. For $n\ge1$, it holds that
\[
q_{n+1}
=
\min\left\{
q_n,
\frac12d(\bsx_{n+1},X_n)
\right\},
\]
with the first term omitted when $n=1$. If $\gamma_n\ge\underline{\gamma}$ for all $n$, then $q_2 \ge \underline{\gamma}h_1/2$. Assume inductively that $q_n \ge \underline{\gamma}h_{n-1}/2$. Since $h_n\le h_{n-1}$, we have
\[
q_{n+1}
\ge
\min\left\{
\frac{\underline{\gamma}}2h_{n-1},
\frac{\underline{\gamma}}2h_n
\right\}
=
\frac{\underline{\gamma}}2h_n,
\]
and so
\[
\rho_{\Xcal}(X_n)
=
\frac{h_n}{q_n}
\le
\frac{h_{n-1}}{q_n}
\le
\frac{2}{\underline{\gamma}}.
\]
This proves \eqref{eq:uniform-relaxed-separation}. 

Suppose now that $\gamma_\infty>0$. Since each $\gamma_n$ is positive, this already implies $\inf_{n\ge1}\gamma_n>0$, and hence the sequence is quasi-uniform by the first part. It remains to prove the sharper asymptotic estimate. Fix $0<\eta<\gamma_\infty$. There exists $n_0$ such that $\gamma_n\ge\gamma_\infty-\eta$ for every $n\ge n_0$. We claim that there exists $n_1\ge n_0$ such that
\begin{equation}
q_{n_1}
>
\frac{\gamma_\infty-\eta}{2}h_{n_1}.
\label{eq:strict-relaxed-stage}
\end{equation}
Otherwise,
\[
q_n
\le
\frac{\gamma_\infty-\eta}{2}h_n
\quad
\text{for every $n\ge n_0$}.
\]
It would then follow that
\[
q_{n+1}
=
\min\left\{
q_n,
\frac12d(\bsx_{n+1},X_n)
\right\}
\ge
\min\left\{
q_n,
\frac{\gamma_\infty-\eta}{2}h_n
\right\}
=
q_n.
\]
Since adding a point cannot increase the separation radius, this would imply $q_{n+1}=q_n$ for every $n\ge n_0$. This is impossible because
\eqref{eq:separation-upper-bound} gives $q_n\le(c_-n)^{-1/s}\to 0$. Thus \eqref{eq:strict-relaxed-stage} holds.

Using \eqref{eq:strict-relaxed-stage} and arguing inductively, we obtain
\[
q_{n+1}
\ge
\frac{\gamma_\infty-\eta}{2}h_n
\quad
\text{for every $n\ge n_1$}.
\]
It follows that
\[
\rho_{\Xcal}(X_{n+1})
\le
\frac{2}{\gamma_\infty-\eta}.
\]
Letting $\eta\downarrow0$ proves \eqref{eq:asymptotic-relaxed-mesh}.
\end{proof}

\begin{remark}
\label{rem:relation-relaxed-pz}
\Cref{prop:asymptotic-relaxed-greedy} is the metric-measure analogue of
the relaxed greedy-packing results in \cite[Theorems~3.6 and~3.7]{PZ23}. We have included the proof because the relaxation factors arising from \eqref{eq:nested-relaxed-greedy} are random and depend on the increasing candidate pools.
\end{remark}

\subsection{Almost-sure control of the candidate covering errors}
\label{subsec:nested-covering-errors}

Let $(\delta_n)_{n\ge1}$ be any sequence satisfying
\begin{equation}
0<\delta_n<1,
\quad
\sum_{n=1}^{\infty}\delta_n<\infty.
\label{eq:summable-failure-probabilities}
\end{equation}
Define
\begin{equation}
L_n
:=
L_{M_n,\delta_n}
=
\max\left\{
1,
\log\left(\frac{2^sM_n}{\delta_n}\right)
\right\}\quad \text{and}\quad 
r_n
:=
2\left(
\frac{L_n}{c_-M_n}
\right)^{1/s}.
\label{eq:nested-Ln}
\end{equation}
The corresponding deterministic upper bound for the relative candidate covering error is
\begin{equation}
\vartheta_n
:=
2\left(
\frac{
c_+nL_n
}{
c_-M_n
}
\right)^{1/s}.
\label{eq:nested-relative-error}
\end{equation}

\begin{lemma}[Almost-sure candidate coverage]
\label{lem:nested-almost-sure-coverage}
Suppose that $(\Xcal,d,\mu)$ is $s$-regular. Then, with probability~$1$, there exists a finite random index $n_0$ such that $\varepsilon_n\le r_n$ and $\varepsilon_n/h_n^*(\Xcal) \le \vartheta_n$ for every $n\ge n_0$.
\end{lemma}

\begin{proof}
By \eqref{eq:candidate-covering-high-probability}, $\PP(\varepsilon_n>r_n)\le\delta_n$. Since the failure probabilities are summable, the first Borel--Cantelli
lemma implies that $\varepsilon_n\le r_n$ for all sufficiently large $n$, almost surely. No independence among these
events is required. Furthermore, \eqref{eq:optimal-fill-bounds} gives $h_n^*(\Xcal) \ge (c_+n)^{-1/s}$. Therefore,
\[
\frac{r_n}{h_n^*(\Xcal)}
\le
2\left(
\frac{
c_+nL_n
}{
c_-M_n
}
\right)^{1/s}
=
\vartheta_n.
\]
This completes the proof.
\end{proof}

We now obtain the principal almost-sure result for the nested construction.

\begin{theorem}[Almost-sure nested quasi-uniformity]
\label{thm:nested-quasi-uniform}
Suppose that $(\Xcal,d,\mu)$ is $s$-regular, let $(\delta_n)_{n\ge1}$ satisfy \eqref{eq:summable-failure-probabilities}, and define $\vartheta_n$ by \eqref{eq:nested-relative-error}. Assume that
\[
    a_\star
    :=
    \liminf_{n\to\infty}
    \alpha_n(1-\vartheta_n)
    >0.
\]
Then \Cref{alg:nested-randomized-fps} generates a quasi-uniform
sequence almost surely. More precisely,
\begin{equation}
\limsup_{n\to\infty}
\rho_{\Xcal}(X_n)
\le
\frac{2}{a_\star}
\quad
\text{almost surely}.
\label{eq:nested-general-limsup}
\end{equation}
Consequently, almost surely, $h_{\Xcal}(X_n)\asymp n^{-1/s}$ and $q(X_n)\asymp n^{-1/s}$.
\end{theorem}

\begin{proof}
We work on the probability-one event of \Cref{lem:nested-almost-sure-coverage}. For all sufficiently large $n$,
\[
\frac{\varepsilon_n}{h_n^*(\Xcal)}
\le
\vartheta_n<1.
\]
Since $h_{\Xcal}(X_n)\ge h_n^*(\Xcal)$, we have
\[
\frac{\varepsilon_n}{h_{\Xcal}(X_n)}
\le
\frac{\varepsilon_n}{h_n^*(\Xcal)}
\le
\vartheta_n.
\]
It follows from \eqref{eq:nested-relaxed-greedy} that
\[
d(\bsx_{n+1},X_n)
\ge
\alpha_n
\left(
1-
\frac{\varepsilon_n}{h_{\Xcal}(X_n)}
\right)
h_{\Xcal}(X_n)
\ge
\alpha_n(1-\vartheta_n)
h_{\Xcal}(X_n).
\]
Thus the effective factors in \eqref{eq:effective-relaxation-factor} satisfy $\liminf_{n\to\infty}\gamma_n \ge a_\star$. The mesh-ratio estimate follows from \Cref{prop:asymptotic-relaxed-greedy}. The optimal orders of $h_{\Xcal}(X_n)$ and $q(X_n)$ follow from \Cref{cor:quasi-uniform-scales}.
\end{proof}

\begin{remark}[Finite initial segments]
\label{rem:nested-finite-initial-segments}
The bound in \eqref{eq:nested-general-limsup} is asymptotic. A small candidate pool at one of the first few iterations may produce a larger mesh ratio. This does not affect quasi-uniformity: almost surely every finite initial mesh ratio is finite, and \eqref{eq:nested-general-limsup} supplies a deterministic asymptotic bound.
\end{remark}

\subsection{Sharp candidate budgets}
\label{subsec:nested-candidate-budgets}

We first show that the logarithmic oversampling factor remains necessary for a nested construction. This follows immediately from the lower bound for arbitrary candidate thinning.

\begin{proposition}[Subcritical nested budgets]
\label{prop:nested-subcritical}
Suppose that $(\Xcal,d,\mu)$ is $s$-regular and that $M_n=o(n\log n)$. Then every construction satisfying $X_{n+1}\subseteq\Ycal_{M_n}$ with $|X_{n+1}|=n+1$ obeys
\begin{equation}
\rho_{\Xcal}(X_{n+1})
\to \infty
\quad
\text{in probability}.
\label{eq:nested-subcritical-divergence}
\end{equation}
In particular, the probability that the resulting infinite sequence is quasi-uniform is $0$.
\end{proposition}

\begin{proof}
Since $X_{n+1}$ is an $(n+1)$-point subset of $\Ycal_{M_n}$, \Cref{cor:necessity-logarithmic-oversampling} gives
\[
\rho_{\Xcal}(X_{n+1})
\to \infty
\quad
\text{in probability}.
\]
This conclusion is uniform over all subset-selection rules, including randomized rules and rules depending on the previously selected points.

For $K\in\NN$, define
\[
A_K
:=
\left\{
\sup_{n\ge2}
\rho_{\Xcal}(X_n)
\le K
\right\}.
\]
For every $n$, $A_K \subseteq \left\{ \rho_{\Xcal}(X_{n+1})\le K \right\}$, and therefore
\[
\PP(A_K)
\le
\liminf_{n\to\infty}
\PP\left(
\rho_{\Xcal}(X_{n+1})\le K
\right)
=
0.
\]
Taking the countable union over $K$ proves the final assertion.
\end{proof}

We next show that a sufficiently large multiple of $n\log n$ is enough. The explicit constant below is not intended to be sharp; it results from the elementary covering estimate and the choice of summable failure probabilities.

\begin{corollary}[Critical candidate budget]
\label{cor:nested-critical-budget}
Suppose that $(\Xcal,d,\mu)$ is $s$-regular. Fix $p>1$ and set $\delta_n:=(n+1)^{-p}$. Let
\begin{equation}
M_n
:=
\max\left\{
n+1,
\left\lceil
Cn\log(n+1)
\right\rceil
\right\},
\label{eq:critical-nested-budget}
\end{equation}
where $C>0$, and define
\begin{equation}
\vartheta_C
:=
2\left(
\frac{
(p+1)c_+
}{
Cc_-
}
\right)^{1/s}.
\label{eq:critical-nested-theta}
\end{equation}
If $\vartheta_C<1$ and $a := \liminf_{n\to\infty}\alpha_n>0$, then \Cref{alg:nested-randomized-fps} generates a quasi-uniform
sequence almost surely and
\begin{equation}
\limsup_{n\to\infty}
\rho_{\Xcal}(X_n)
\le
\frac{2}{
a(1-\vartheta_C)
}
\quad
\text{almost surely}.
\label{eq:critical-nested-mesh-bound}
\end{equation}
\end{corollary}

\begin{proof}
For this choice of $\delta_n$ and $M_n$, we have $L_n=(p+1+o(1))\log n$ and $M_n=(C+o(1))n\log n$, and therefore $\vartheta_n\to\vartheta_C$. The result follows from \Cref{thm:nested-quasi-uniform}.
\end{proof}

Combining \Cref{prop:nested-subcritical} and \Cref{cor:nested-critical-budget} shows $M_n\asymp n\log n$ is the sharp candidate-budget order for constructing a quasi-uniform nested sequence from an iid stream. The constant multiplying $n\log n$ must be sufficiently large for the upper result, whereas every subcritical budget fails almost surely.

A supercritical candidate budget recovers the asymptotically optimal uniformity constant.

\begin{theorem}[Supercritical nested construction]
\label{thm:nested-supercritical}
Suppose that $(\Xcal,d,\mu)$ is $s$-regular and that
\begin{equation}
\frac{M_n}{
n\log(n+1)
}
\to \infty.
\label{eq:nested-supercritical-budget}
\end{equation}
If $a := \liminf_{n\to\infty}\alpha_n>0$, then \Cref{alg:nested-randomized-fps} generates a quasi-uniform
sequence almost surely and
\begin{equation}
\limsup_{n\to\infty}
\rho_{\Xcal}(X_n)
\le
\frac{2}{a}
\quad
\text{almost surely}.
\label{eq:nested-supercritical-bound}
\end{equation}
In particular, if $\alpha_n\to 1$, then
\begin{equation}
\limsup_{n\to\infty}
\rho_{\Xcal}(X_n)
\le2
\quad
\text{almost surely}.
\label{eq:nested-limsup-at-most-two}
\end{equation}
\end{theorem}

\begin{proof}
Fix any $p>1$ and take $\delta_n=(n+1)^{-p}$. Put $A_n:=M_n/n$. Since
$A_n/\log(n+1)\to\infty$ and
$\log A_n/A_n\to0$, we have
\[
    \frac{nL_n}{M_n}
    =
    O\left(
        \frac{\log n+\log A_n}{A_n}
    \right)
    \longrightarrow0.
\]
Hence $\vartheta_n\to0$, and the result follows from
\Cref{thm:nested-quasi-uniform}.
\end{proof}

\begin{corollary}[Optimal limiting constant on Euclidean domains]
\label{cor:nested-optimal-constant}
In addition to the assumptions of \Cref{thm:nested-supercritical}, suppose that $\Xcal\subset\RR^d$ is compact with positive $d$-dimensional
Lebesgue measure and that the metric is induced by a norm on $\RR^d$.
If $\alpha_n\to1$, then
\begin{equation}
\limsup_{n\to\infty}
\rho_{\Xcal}(X_n)
=
2
\quad
\text{almost surely}.
\label{eq:nested-optimal-limsup}
\end{equation}
\end{corollary}

\begin{proof}
The upper bound follows from \Cref{thm:nested-supercritical}. On the other hand, \cite[Theorem~1.1]{PZ23} states that every nested sequence in a compact positive-volume subset of $\RR^d$ satisfies
\[
\limsup_{n\to\infty}
\rho_{\Xcal}(X_n)\ge2.
\]
This lower bound applies pathwise to every realization of the randomized construction. Combining the two inequalities proves the result.
\end{proof}

\begin{remark}[Examples of supercritical budgets]
\label{rem:nested-budget-examples}
The following candidate budgets satisfy \eqref{eq:nested-supercritical-budget}: $M_n=n^{1+\beta}$ and $M_n=n(\log n)^{1+\beta}$ with some $\beta>0$. For polynomially small $\delta_n$, the corresponding deterministic relative-error bounds have the orders
\[
\begin{array}{rcl}
M_n=n^{1+\beta}
&:&
\displaystyle
\vartheta_n
=
O\left[
\left(
\frac{\log n}{n^\beta}
\right)^{1/s}
\right],
\\[3mm]
M_n=n(\log n)^{1+\beta}
&:&
\displaystyle
\vartheta_n
=
O\left(
(\log n)^{-\beta/s}
\right).
\end{array}
\]
Thus the original choice of a quadratic candidate pool is more than sufficient, while the nearly critical choice $n(\log n)^{1+\beta}$ already yields the optimal limiting mesh-ratio constant.
\end{remark}

\section{Numerical experiments}
\label{sec:numerics}

We numerically investigate the principal predictions of the theoretical analysis: the iid covering-radius scale $(\log M/M)^{1/s}$, the critical candidate-pool order $M\asymp N\log N$, the quality--cost trade-off for approximate FPS, and the behavior of the nested construction under subcritical, critical, and supercritical candidate budgets. All domains considered below are two-dimensional subsets of $\RR^2$ equipped with the Euclidean metric and normalized area measure, so $s=2$ throughout. The first four experiments are conducted on the unit square $\Xcal=[0,1]^2$, while the final experiment additionally considers two nonrectangular domains.

All FPS constructions are initialized at the first candidate in the generated stream. Within each repetition, the same underlying candidate stream is used whenever several candidate budgets or algorithms are compared, thereby yielding paired comparisons. For a point set in the unit square, the separation radius is computed directly from the pairwise distances. The covering radius is computed, up to floating-point roundoff, by augmenting the point set with its $3\times3$ reflected copies across the sides of the square, constructing the Delaunay triangulation of the reflected set, and examining the circumcenters lying in $[0,1]^2$ together with the four corners. Unless otherwise stated, solid curves represent empirical medians, and shaded regions represent the empirical $10\%$--$90\%$ quantile bands.

All numerical experiments were performed in MATLAB Online using MATLAB R2026a (The MathWorks, Inc.).

\subsection{Random covering and the candidate-pool threshold} \label{subsec:numerical-threshold} 
We first examine the covering radius $\varepsilon_M=h_{\Xcal}(\Ycal_M)$ of an iid candidate cloud. For $M=2^6,2^7,\ldots,2^{13}$, we generate $30$ independent candidate sets and plot the normalized quantity $\varepsilon_M\sqrt{M/\log M}$ in \Cref{fig:covering-scaling}. Its empirical median remains between approximately $0.72$ and $0.85$ over the entire range of $M$, while the empirical quantile band remains bounded. This is consistent with the random-covering order $\varepsilon_M \asymp_{\PP} \left(\log M/M\right)^{1/2}$ proven in \Cref{cor:candidate-covering-lower}.

\begin{figure}[t] 
\centering 
\includegraphics[width=0.6\textwidth]{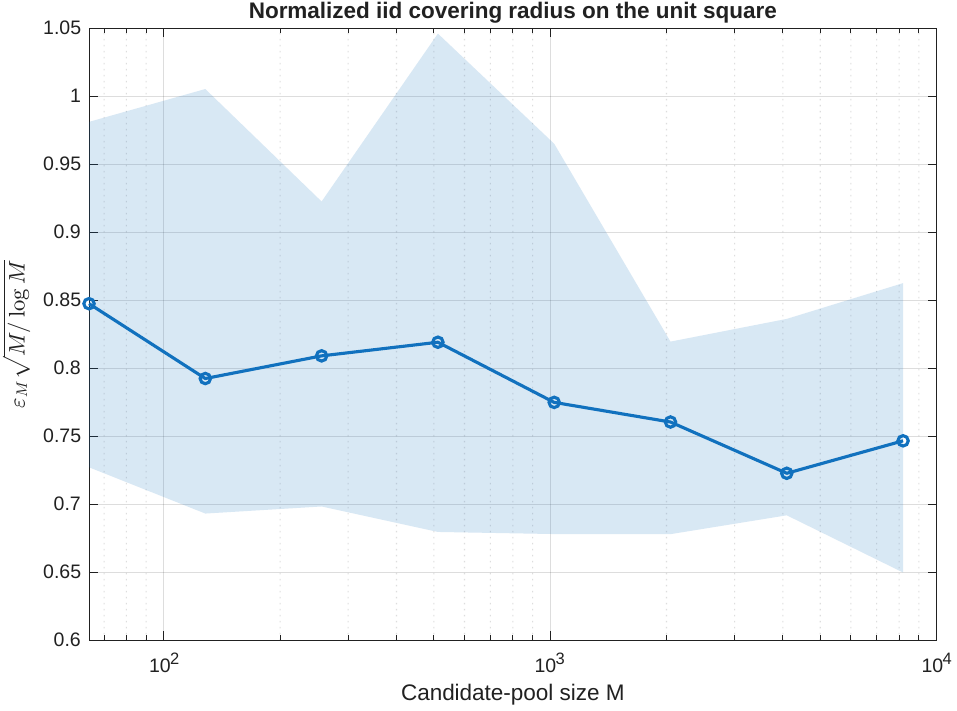} 
\caption{The normalized covering radius $\varepsilon_M\sqrt{M/\log M}$ of $M$ iid uniform points in $[0,1]^2$. The solid curve is the empirical median over $30$ independent repetitions, and the shaded region is the empirical $10\%$--$90\%$ quantile band.} 
\label{fig:covering-scaling} 
\end{figure} 

We next study the candidate-pool threshold directly. We take $N\in\{32,64,128,256,512\}$ and compare the six schedules 
\[ M=N,\, 2N,\, N\sqrt{\log N},\, N\log N,\, 4N\log N,\, N(\log N)^{3/2}.\] 
For each $N$, the largest required iid stream is generated first, and its initial segments are used for the smaller candidate pools. Exact FPS is then applied to select $N$ points. Each setting is repeated $30$ times. 

The left panel of \Cref{fig:candidate-threshold} shows the resulting mesh ratios. With $M=N$, no thinning takes place, and the median mesh ratio grows from approximately $28.5$ at $N=32$ to $151.6$ at $N=512$. The subcritical schedule $M=2N$ exhibits a gradual increase, from approximately $4.30$ to $5.05$. The growth predicted for $M=N\sqrt{\log N}$ is only of logarithmic order and is not yet clearly visible over the present range. In contrast, the critical and supercritical schedules remain uniformly small: the median mesh ratios at $N=512$ are approximately $3.02, 2.38, 2.48$ for $M=N\log N$, $M=4N\log N$, and $M=N(\log N)^{3/2}$, respectively. 

\begin{figure}[t] 
\centering 
\includegraphics[width=\textwidth]{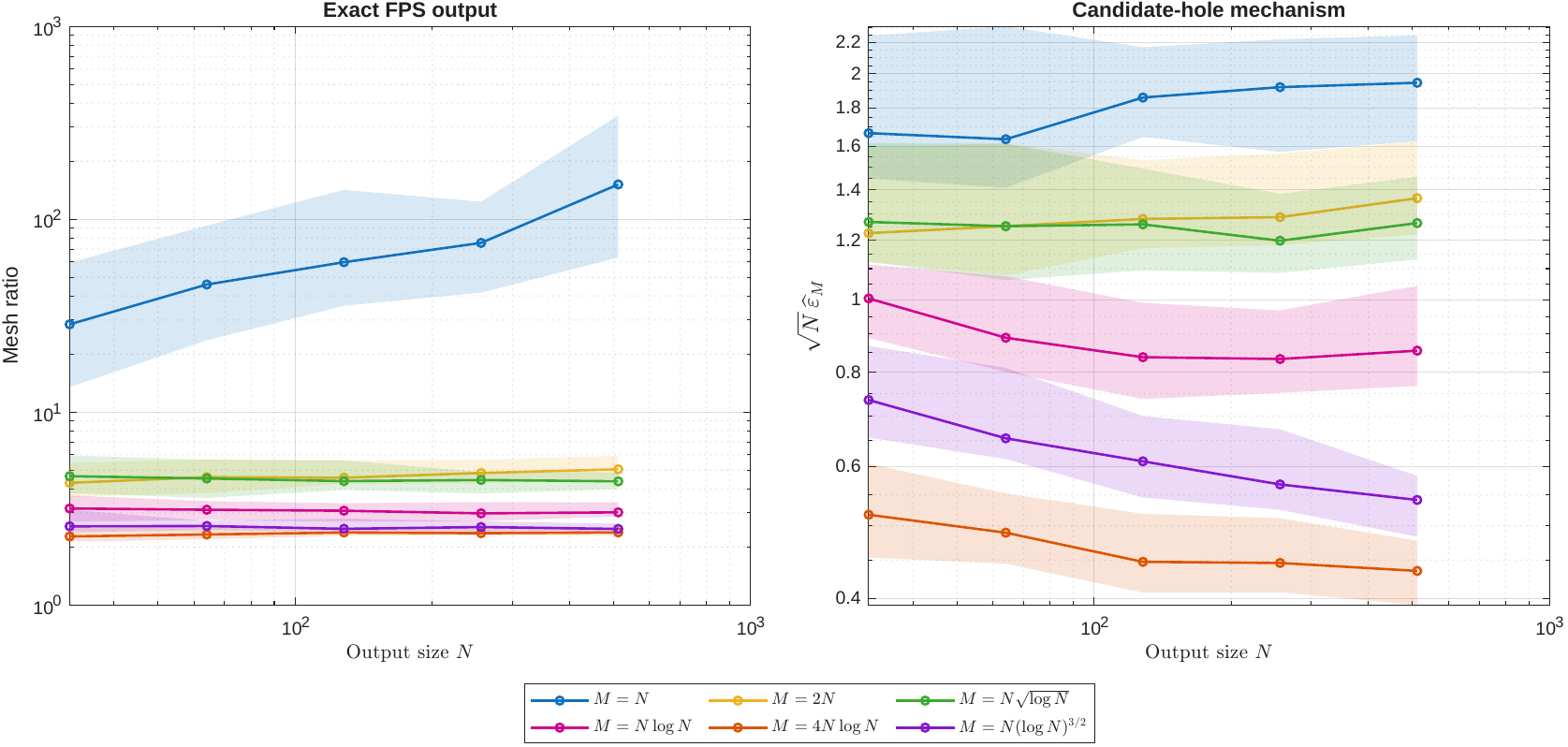} 
\caption{ Candidate-pool threshold for exact FPS on $[0,1]^2$. Left: mesh ratio of the selected $N$-point set. Right: the candidate-hole quantity $\sqrt N\,\widehat{\varepsilon}_M$, with $\widehat{\varepsilon}_M$ evaluated on a $401\times401$ validation grid. Curves and shaded regions show empirical medians and $10\%$--$90\%$ quantile bands over $30$ repetitions. } 
\label{fig:candidate-threshold} 
\end{figure} 

To visualize the mechanism behind the lower bound of \Cref{thm:arbitrary-thinning-lower}, the right panel plots $\sqrt{N}\,\widehat{\varepsilon}_M$, where $\widehat{\varepsilon}_M$ is the candidate covering radius evaluated on a $401\times401$ validation grid. Since $\sqrt{N}\,\varepsilon_M \asymp_{\PP}\left((N\log M)/M\right)^{1/2}$, this quantity is expected to grow, although possibly very slowly, for the three subcritical schedules, to remain of constant order when $M\asymp N\log N$, and to decrease under supercritical oversampling. The numerical results are broadly consistent with these predictions. The curve for $M=N$ is the largest and shows a clear upward tendency, whereas the theoretically slower growth for $M=2N$ and $M=N\sqrt{\log N}$ is difficult to distinguish over the present range of~$N$. The two critical schedules remain approximately of constant order, with the larger candidate pool $M=4N\log N$ producing substantially smaller holes than $M=N\log N$, while the supercritical curve $M=N(\log N)^{3/2}$ decreases overall.

Over the tested range, $N(\log N)^{3/2}$ is still smaller than $4N\log N$. Hence, the finite-sample ordering of these two curves need not coincide with their asymptotic classification as critical and supercritical schedules. 

\subsection{Quality and cost of approximate FPS} \label{subsec:numerical-approximate} 

We now compare direct exact FPS with the multiscale method of \Cref{alg:multiscale-fps}. We use $N\in\{64,128,256,512,1024\}$, $M=\left\lceil4N\log N\right\rceil$, and $\beta\in \{1.02, 1.05, 1.10, 1.25, 1.50, 2.00\}$. Each experiment is repeated $15$ times using the same candidate cloud for all values of $\beta$. By the interpretation in \Cref{subsec:metric-net-mis}, the same experiment also illustrates the hierarchical-MIS viewpoint of the multiscale construction at a prescribed output cardinality. 

The upper-left panel of \Cref{fig:approximate-quality-cost} shows the mesh ratio at $N=1024$. The empirical median is approximately $2.39$ for exact FPS and increases to $2.42, 2.51, 2.62, 2.85, 3.48, 4.46$ as $\beta$ ranges over the six values above. Thus, the geometric quality deteriorates smoothly as the farthest-point requirement is relaxed. The observed increase relative to exact FPS is modest for moderate values of $\beta$. \Cref{eq:multiscale-mesh-bound} multiplies the common deterministic upper bound by $\beta$; it does not provide a pointwise comparison with the mesh ratio of the exact FPS output on the same candidate set.

\begin{figure}[t] 
\centering 
\includegraphics[width=\textwidth]{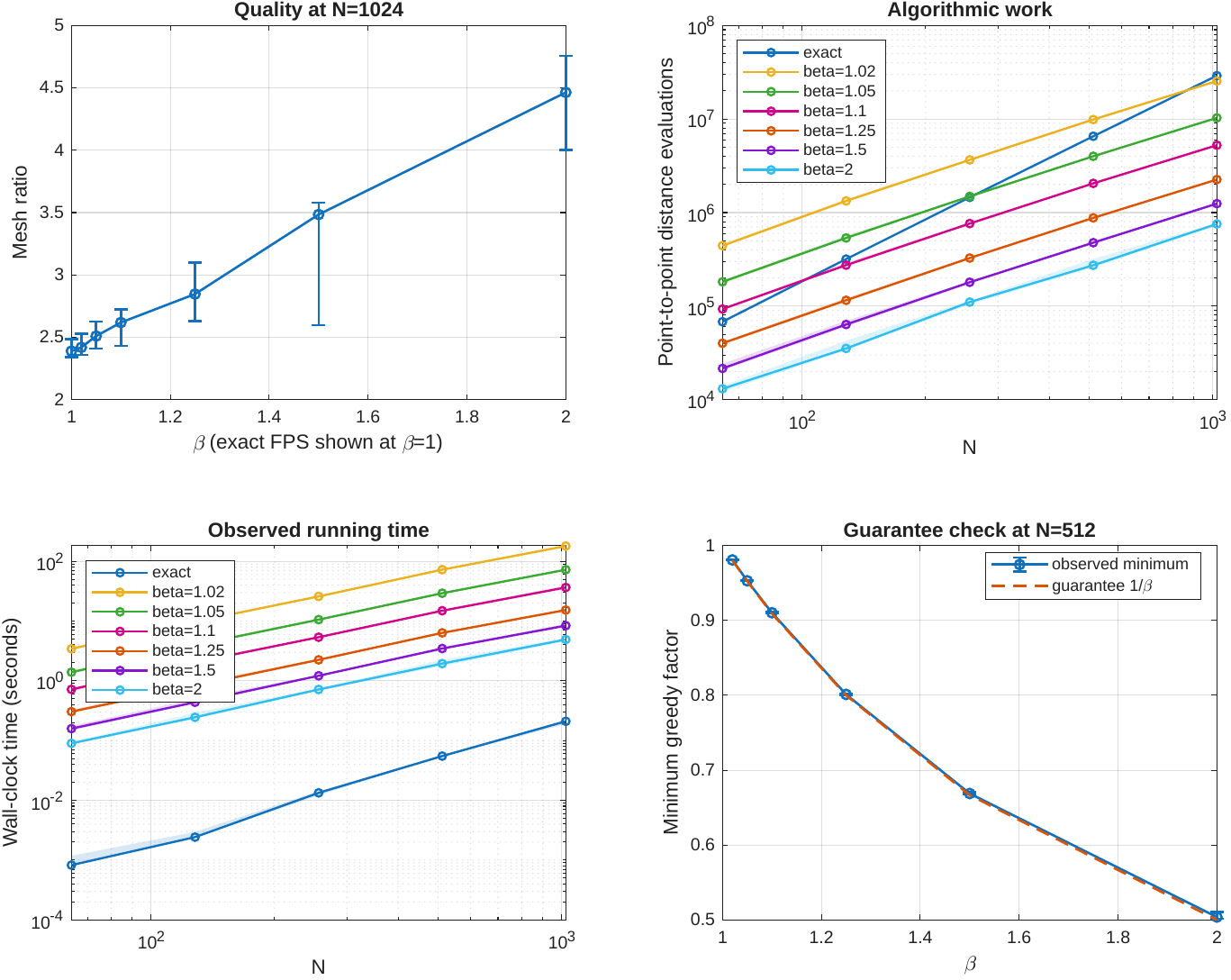} 
\caption{Quality and cost of multiscale approximate FPS with $M=\lceil4N\log N\rceil$. Upper left: mesh ratio at $N=1024$, with exact FPS displayed at $\beta=1$. Upper right: number of point-to-point distance evaluations. Lower left: observed MATLAB running time. Lower right: empirical minimum greedy factor at $N=512$ compared with the guarantee $1/\beta$. Medians and empirical $10\%$--$90\%$ ranges are computed from $15$ repetitions. } 
\label{fig:approximate-quality-cost} 
\end{figure} 

The upper-right panel reports the number of point-to-point distance evaluations. The direct exact implementation uses exactly $MN$ evaluations. By contrast, the multiscale method repeatedly scans the candidate cloud over $K$ geometric levels, where $K$ increases rapidly as $\beta$ approaches~$1$. Consequently, for small $N$ and $\beta$ close to~$1$, the overhead of the repeated multiscale scans can exceed that of direct exact FPS, as seen in the figure. This is a finite-size crossover and is consistent with the complexity bounds: the multiscale method becomes advantageous only when the reduction from $N$ full update sweeps to a much smaller number of geometric levels outweighs the per-level overhead.

For larger $\beta$, the reduction is already substantial. At $N=1024$, direct exact FPS uses approximately $2.91\times10^7$ evaluations, whereas the corresponding medians for $\beta=1.25$, $1.50$, and $2.00$ are approximately $2.25\times10^6$, $1.25\times10^6$, and $7.55\times10^5$, respectively. These represent reductions by factors of approximately $13$, $23$, and $38$.

The lower-right panel verifies the approximation guarantee directly. For every approximate ordering with $N\le512$, we replay the ordering and compute 
\[ \widehat{\alpha}_{\min} := \min_{1\le n<N} \frac{d(\bsx_{n+1},X_n)} {h_{\Ycal_M}(X_n)}. \] 
The observed values lie almost exactly on, and always above, the theoretical lower bound $1/\beta$. This provides a direct numerical check of \eqref{eq:multiscale-approximation}. 

The lower-left panel reports wall-clock times measured in MATLAB Online. These do not mirror the distance-evaluation counts: the direct exact routine is heavily vectorized, whereas the multiscale prototype uses repeated scalar neighbor queries and MATLAB hash-table operations. Consequently, direct exact FPS is faster throughout the tested range in our MATLAB implementation. Moreover, since MATLAB Online runs on cloud-based computational resources without a fixed hardware configuration, the reported wall-clock times should be regarded as implementation-specific indicative values rather than hardware-independent benchmarks. The number of point-to-point distance evaluations provides the more reproducible measure of algorithmic work. A lower-level implementation would be needed to determine a practical crossover point. We also emphasize that the more sophisticated fast exact algorithm of \cite{HM06} is not implemented here. 

\subsection{Nested randomized construction} \label{subsec:numerical-nested} 
We next test the nested construction of \Cref{sec:nested}. In each of $20$ repetitions, we generate one infinite iid stream up to the largest required index and construct exact nested FPS sequences with the three candidate-budget schedules 
\[ M_n=2n,\, 4n\log(n+1),\, 2n\{\log(n+1)\}^{3/2}. \] 
The mesh ratio and the covering and separation radii are evaluated at $36$ logarithmically spaced prefix lengths between $4$ and $2048$. 

The left panel of \Cref{fig:nested-sequence} clearly separates the subcritical schedule from the two logarithmically oversampled schedules. For $M_n=2n$, the median mesh ratio increases overall and equals approximately $5.70$ at $n=2048$. In contrast, the critical and supercritical schedules remain bounded over the tested range: their median mesh ratios at $n=2048$ are approximately $2.46$ and $2.35$, respectively. The supercritical curve does not show a clear tendency toward~$2$ over the present range of prefix lengths. This is not inconsistent with \Cref{thm:nested-supercritical} and \Cref{cor:nested-optimal-constant}, which concern the asymptotic limit superior rather than monotone convergence of the mesh ratio. Moreover, for the schedule $M_n=2n\{\log(n+1)\}^{3/2}$, the oversampling ratio $M_n/[n\log(n+1)]=2\sqrt{\log(n+1)}$ diverges only slowly. In the present two-dimensional setting, the corresponding deterministic relative-error estimate from \Cref{rem:nested-budget-examples} decays only as $O((\log n)^{-1/4})$. We therefore interpret the experiment as evidence of quasi-uniform behavior under critical and supercritical candidate budgets, but not as numerical verification of the optimal asymptotic constant~$2$.

\begin{figure}[t] 
\centering 
\includegraphics[width=\textwidth]{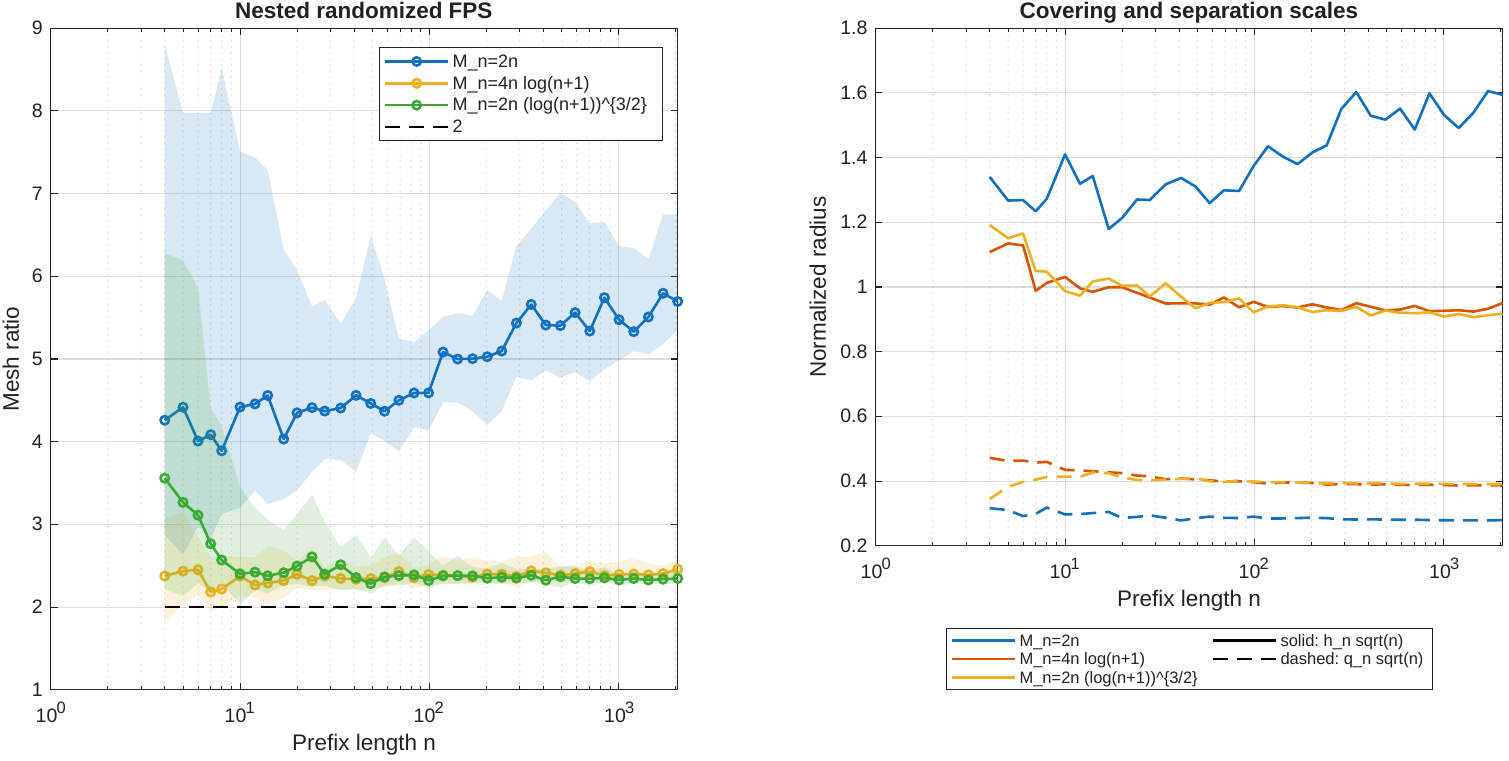} 
\caption{Nested exact FPS on $[0,1]^2$ under subcritical, critical, and supercritical candidate budgets. Left: mesh ratio; solid curves and shaded regions are empirical medians and $10\%$--$90\%$ quantile bands over $20$ repetitions. Right: median scaled covering radii $h_{\Xcal}(X_n)\sqrt n$ (solid) and separation radii $q(X_n)\sqrt n$ (dashed). } 
\label{fig:nested-sequence} 
\end{figure} 

The right panel displays the medians of $h_{\Xcal}(X_n)\sqrt n$ and $q(X_n)\sqrt n$. For the critical and supercritical schedules, both scaled quantities stabilize: at $n=2048$, their median pairs are approximately $(0.95,0.39)$ and $(0.92,0.39)$, respectively. This is the expected quasi-uniform behavior $h_{\Xcal}(X_n)\asymp q(X_n)\asymp n^{-1/2}$. For the subcritical schedule, the corresponding pair is approximately $(1.59,0.28)$: the covering scale becomes substantially larger while the separation scale remains smaller, producing the increasing mesh ratio in the left panel. 


\subsection{Examples on nonrectangular domains} \label{subsec:numerical-domains} Finally, we illustrate the geometry-independent character of the construction on the unit square, the L-shaped domain $[0,1]^2\setminus(1/2,1]\times(1/2,1]$, and the annulus $\{\bsx\in\RR^2 \mid 0.35\le\|\bsx\|_2\le1\}$. For each domain, we generate $M=\left\lceil4N\log N\right\rceil=12430$ candidates and apply exact FPS with $N=500$. The candidates are sampled uniformly with respect to the area measure. 

The resulting designs are shown in \Cref{fig:domain-examples}. The mesh ratio on the square, whose covering radius is computed exactly, is approximately $2.536$. For the L-shaped domain and annulus, the covering radii are evaluated on deterministic validation sets of approximately $150\,000$ points, giving estimated mesh ratios $\widehat{\rho}\approx2.442$ and $2.198$, respectively. Since a finite validation set can underestimate the true covering radius, the latter two values should be interpreted as numerical estimates. More precisely, the validation sets cover the L-shaped domain and the annulus with radii at most $1.59\times10^{-3}$ and $5.18\times10^{-3}$, respectively. Since the function $\bsx\mapsto d(\bsx,X_N)$ is $1$-Lipschitz, the true covering radius of each design exceeds its validation-set estimate by at most the corresponding amount. The true mesh ratios are therefore bounded above by $2.547$ and $2.375$, respectively. The point configurations show that the same sampling-and-FPS procedure adapts directly to nonconvex and multiply connected domains, without requiring a continuous largest-empty-ball solver or an explicit Voronoi construction for the domain. 

\begin{figure}[t] 
\centering 
\includegraphics[width=0.85\textwidth]{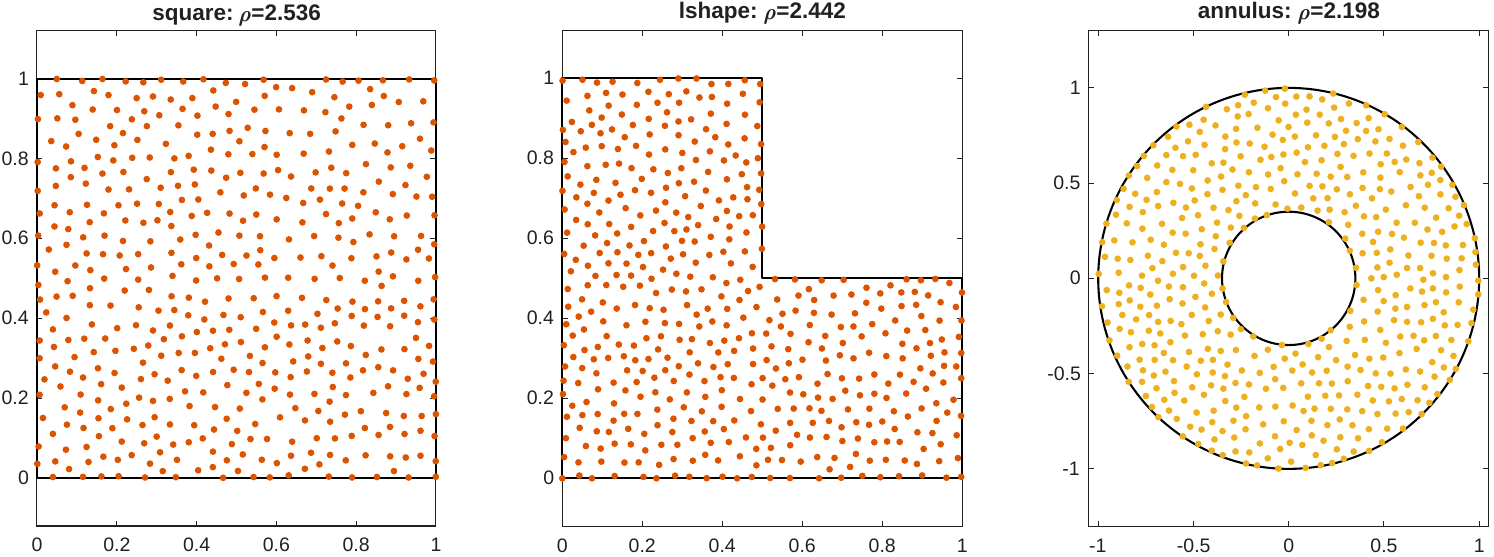} 
\caption{Exact FPS designs with $N=500$ selected from $M=12430$ uniform iid candidates. From left to right: the unit square, an L-shaped domain, and an annulus with inner radius $0.35$. The square mesh ratio is computed exactly; the values displayed for the L-shaped domain and annulus use dense validation sets and are therefore numerical estimates. } 
\label{fig:domain-examples} 
\end{figure} 

\section{Concluding remarks} \label{sec:conclusion} 

We have studied a simple randomized strategy for constructing quasi-uniform designs on compact metric-measure spaces. The method first generates an iid candidate set and then selects a prescribed number of points by farthest-point sampling. Under a two-sided polynomial ball-growth condition, we showed that the candidate-pool size $M=\Theta(N\log N)$ is the sharp order required for obtaining an $N$-point set with bounded mesh ratio. The upper bound is achieved by exact FPS with high probability, whereas the matching lower bound applies to every procedure that is constrained to select its output from the same iid candidate cloud. Thus, the logarithmic oversampling factor is a consequence of the random covering properties of iid samples rather than of the greedy selection rule itself. 

We also quantified the effect of inexact farthest-point searches and showed that the error caused by approximate FPS can be separated from the discretization error of the random candidate set. This leads naturally to fast implementations. In particular, in spaces of bounded doubling dimension, existing greedy-permutation algorithms yield near-linear complexity in the candidate-pool size, while a simple multiscale approximate FPS admits an efficient grid-based implementation on fixed-dimensional Euclidean domains. Hence, at the critical candidate size, quasi-uniform point sets can be constructed with computational cost of order $N\log^2 N$ under the assumptions considered in this paper. We also placed exact and approximate FPS in a common metric-net framework, which identifies the completed levels of the multiscale construction with maximal independent sets and clarifies the complementary fixed-cardinality and fixed-scale viewpoints.

A further advantage of the proposed approach is that it extends naturally to nested constructions. By using a single infinite iid stream together with increasing candidate budgets, we obtained an almost surely quasi-uniform infinite sequence. The candidate-budget order $n\log n$ is again sharp for quasi-uniformity, while supercritical oversampling combined with asymptotically exact FPS gives $\limsup_{n\to\infty}\rho_{\Xcal}(X_n)\le 2$ almost surely. For compact positive-volume subsets of Euclidean space, this upper bound is optimal. 

Several questions remain open. First, although the order $N\log N$ is sharp, the behavior in the critical regime $M=C N\log N$ deserves a more precise analysis. In particular, it would be interesting to determine how the best achievable mesh-ratio bound depends on the oversampling constant $C$, and whether sharp threshold constants can be identified for specific domains. Second, the logarithmic factor is specific to iid candidate generation. Candidate sets obtained from stratified, randomized low-discrepancy, or other well-covering sampling schemes may have a covering radius of order $M^{-1/s}$ and could therefore reduce the required candidate size to order $N$. Developing a general theory that relates candidate-generation quality to the performance of subsequent FPS thinning would be worthwhile. Finally, quasi-uniformity is often used as a geometric condition for stable and accurate scattered-data approximation, kernel interpolation, meshless methods, and computer experiments. Combining the probabilistic construction developed here with quantitative approximation or prediction error bounds would provide a direct link between candidate complexity, computational cost, and downstream numerical performance. Such extensions may also be useful for sequential experimental design, where the nested construction is particularly natural.

\section*{Funding}
The work of T.G.\ was supported by JSPS KAKENHI Grant Number JP26K00620.

\section*{Code and data availability}

The MATLAB code, random seeds, and raw numerical outputs used to generate the figures will be made publicly available upon acceptance.

\section*{Declaration of AI use}

The authors used OpenAI's ChatGPT to assist with the drafting and editing of portions of the manuscript and with checking mathematical arguments and presentation. All AI-assisted content was critically reviewed, verified, and revised by the authors, who take full responsibility for the accuracy and integrity of the manuscript.

\bibliographystyle{siam}
\bibliography{ref}

\end{document}